\documentclass[12pt]{amsart}
\usepackage{amsmath}
\usepackage{amsthm}
\usepackage{amsfonts} 
\usepackage{hyperref}
\usepackage{graphicx}

\numberwithin{equation}{section}

\allowdisplaybreaks

\newtheorem{theorem}{Theorem}[section]
\newtheorem{proposition}{Proposition}[section]
\newtheorem{lemma}{Lemma}[section]

\renewcommand{\epsilon}{\varepsilon}

\newcommand{\abs}[1]{\left\vert #1\right\vert}

\newcommand{\R}{\mathbb{R}}
\newcommand{\N}{\mathbb{N}}
\newcommand{\Z}{\mathbb{Z}} 
\newcommand{\PA}{\mathbb{P}}
\newcommand{\EA}{\mathbb{E}}

\newcommand{\tom}{{\tilde\omega}}
\newcommand{\Po}{P_{\omega}}
\newcommand{\Eo}{E_{\omega}}
\newcommand{\Pto}{P_{\tom}}
\newcommand{\Eto}{E_{\tom}}
\newcommand{\PP}{\mathbf{P}}
\newcommand{\EE}{\mathbf{E}}
\newcommand{\PAt}{{\tilde{\mathbb{P}}}}
\newcommand{\EAt}{{\tilde{\mathbb{E}}}}

\newcommand{\hdel}{{\delta}}
\newcommand{\ts}{{\tilde\sigma}}
\newcommand{\card}{\mathop{\mathrm{card}}}
\newcommand{\Tinit}{{\mathcal T}_{\text{\itshape{init}}}}
\newcommand{\Tdir}{{\mathcal T}_{\text{\itshape{dir}}}}
\newcommand{\Tback}{{\mathcal T}_{\text{\itshape{back}}}}
\newcommand{\Tleft}{{\mathcal T}_{\text{\itshape{left}}}}
\newcommand{\Tright}{{\mathcal T}_{\text{\itshape{right}}}}

\newcommand{\Bk}{{\mathfrak{B}}}
\newcommand{\II}{{\mathcal I}}
\newcommand{\UU}{{\mathcal U}}
\newcommand{\WW}{{\mathcal W}}
\newcommand{\RR}{{\mathcal R}}

\newcommand{\1}[1]{{\mathbf 1}{\{#1\}}}

\author[A.~Fribergh]{Alexander FRIBERGH}
\address{Universit\'e de Lyon;
Universit\'e Lyon 1;
INSA de Lyon, F-69621;
Ecole Centrale de Lyon;
CNRS, UMR5208, Institut Camille Jordan,
43 blvd du 11 novembre 1918,
F-69622 Villeurbanne-Cedex, France} 
\email{fribergh@math.univ-lyon1.fr}

\author[N.~Gantert]{Nina GANTERT}
\address{CeNos, Center for Nonlinear Science, and Institut f\"ur Mathematische Statistik, Fachbereich Mathematik und Informatik, Universit\"at M\"unster, Einsteinstr. 62, D--48149 M\"unster, Germany} \email{gantert@math.uni-muenster.de}

\author[S.~Popov]{Serguei POPOV}
\address{Instituto de Matem\'atica e Estat\'istica, Universidade de S\~ao Paulo, rua do Mat\~ao 1010, CEP 05508--090, S\~ao Paulo SP, Brasil} 
\email{popov@ime.usp.br}


\begin{document}

\title[Slowdown and speedup of transient RWRE]{On slowdown and speedup of transient random walks in random environment}

\begin{abstract}
We consider one-dimensional random walks in random environment which are transient to the right.
Our main interest is in the study of
 the sub-ballistic regime, where at time~$n$ the particle is typically at a distance of order $O(n^\kappa)$
from the origin, $\kappa\in(0,1)$. We investigate the probabilities of moderate deviations from
this behaviour. Specifically, we are interested in quenched and annealed probabilities of 
slowdown (at time~$n$, the particle is at a distance of order $O(n^{\nu_0})$ from the origin, $\nu_0\in (0,\kappa)$), and
speedup (at time~$n$, the particle is at a distance of order $n^{\nu_1}$ from the origin, $\nu_1\in (\kappa,1)$),
for the current location of the particle and for the hitting times. Also, we study
probabilities of backtracking: at time~$n$, the particle is located around $(-n^\nu)$,
thus making an unusual excursion to the left. 
For the slowdown, our results are valid in the ballistic case as well.
\end{abstract}

\maketitle

\section{Introduction and results}
\label{s_introres}

Let $\omega:=(\omega_i, \, i \in \Z)$ be a family of i.i.d.\ random variables taking values 
in $(0,1)$.
Denote by $\PP$ the distribution of~$\omega$ and by~$\EE$ the corresponding expectation. 
After choosing an environment~$\omega$ at random according to the law~$\PP$, we define the random walk in random environment (usually abbreviated as RWRE) as a nearest-neighbour random walk on~$\Z$ with transition probabilities given by $\omega$: $(X_n, \, n \geq 0)$ is the Markov chain satisfying $X_0=z$ 
and for $n \geq 0,$
\begin{align*}
\Po^z[ X_{n+1} = x+1  \mid  X_n =x] &= \omega_x, \nonumber\\
\Po^z[ X_{n+1} = x-1  \mid  X_n =x] &= 1-\omega_x.
\end{align*}
As usual, $\Po^z$ is called the \emph{quenched} law of $(X_n, \, n \geq 0)$ starting
 from $X_0=z$, 
and we denote by $\Eo^z$ the corresponding quenched expectation. Also, we denote by~$\PA^z$ 
the semi-direct product $\PP \times \Po^z$ and by~$\EA^z$ the expectation with respect to~$\PA^z$;
$\PA^z$ and~$\EA^z$ are called the \emph{annealed} probability and expectation.
When $z=0$, we write simply $\Po$, $\Eo$, $\PA$, $\EA$.

In this paper we will also consider RWRE on~$\Z_+$, with reflection to the right at the origin.
This RWRE can be defined as above, in the environment~$\tom$
given by
\[
  \tom_i = \begin{cases}
                \omega_i, & i\neq 0,\\ 
                    1, & i=0
             \end{cases}
\]
(provided, of course, that the starting point is nonnegative). We then write
$\Pto^z$, $\Eto^z$ for the quenched probability and expectation in the case
of RWRE reflected at the origin, $\PAt^z$ and~$\EAt^z$ for the annealed probability and expectation,
keeping the simplified notation $\Pto$, $\Eto$, $\PAt$, $\EAt$ for the RWRE
starting at the origin.

For all $i \in \Z$, let us introduce
\[
\rho_i:=\frac{1-\omega_i}{\omega_i}.
\]
Throughout this paper, we assume that 
\begin{equation}
\label{logrho}
\EE[\ln\rho_0]<0,
\end{equation}
 which implies (cf.~\cite{Solomon}) that
$\lim_{n\to\infty} X_n=+\infty$ $\Po$-a.s.\ for $\PP$-a.a.\ $\omega$, 
so that the RWRE is transient to the right
(or simply transient, in the case of RWRE with reflection at the origin).

We refer to~\cite{Zeitouni} for a general overview of results on RWRE. In the following we always
work under the assumption that
\begin{equation}
\label{def_kappa}
\text{there exists a unique $\kappa>0$, such that } \EE[\rho_0^{\kappa}]=1 \text{ and } \EE[\rho_0^{\kappa}\ln^+ \rho_0]<\infty.
\end{equation}

This constant plays a central role for RWRE, in particular when it exists, its value separates the \emph{ballistic} from the  \emph{sub-ballistic} regime:
\[
\kappa >1 \text{ if and only if } \frac{X_n}n \to v>0,\quad \text{$\PA$-a.s.}
\]
We refer to the case $\kappa > 1$ as the \emph{ballistic} regime and to the case $\kappa \leq 1$ as the \emph{sub-ballistic} regime.
In this paper we mainly consider the case where  the RWRE is transient (to the right) and \emph{sub-ballistic}, i.e.~the asymptotic speed is equal to~$0$.
The following result was proved in~\cite{KKS} and partially refined in \cite{ESZ}:
\begin{theorem}
\label{kks}
Let $\omega:=(\omega_i, \, i \in \Z)$ be a family
of independent and identically distributed random variables such
that
\begin{itemize}
 \item[(i)] $-\infty \leq \EE[\ln \rho_0]<0$,
 \item[(ii)] there exists $0<\kappa\leq1$ for which $\EE \left[\rho_0^{\kappa}\right] = 1$ 
  and $\EE \left[  \rho_0^{\kappa} \ln^+\rho_0  \right]<\infty,$
 \item [(iii)] the  distribution  of $\ln \rho_0$ is non-lattice.
\end{itemize}
Then, if $\kappa<1$, we have 
\[
\frac{X_n}{n^{\kappa}} \,
\stackrel{law}{\longrightarrow}\, C_1 \left(\frac 1{\mathcal{S}_{\kappa}^{ca}}\right)^{\kappa},
\]
where $\stackrel{law}{\longrightarrow}$ stands for convergence in distribution with respect to the annealed law $\PA$,
$C_1$ is a positive constant and $\mathcal{S}_{\kappa}^{ca}$ is the completely asymmetric stable law of index $\kappa$. If $\kappa=1$, we have
\[
\frac{X_n}{n/\ln n} \,
\stackrel{law}{\longrightarrow}\,
 C_2 \frac 1{\mathcal{S}_{1}^{ca}}.
\]
\end{theorem}

In the quenched case, the limiting behaviour is more complicated, as discussed in~\cite{PZ}. However, one still can say
that at time~$n$ the particle is ``typically'' at distance roughly~$n^\kappa$
 from the origin, since the weaker result $\lim_{n\to\infty}\ln X_n/\ln n = \kappa$, $\PA$-a.s.,
is still valid\footnote{apparently, this result is folklore, at least we were unable to find
a precise reference in the literature. Anyhow, note that it is straightforward to obtain
this result from Theorems~\ref{t_q_slow} and~\ref{t_speedup}}.

Besides the results about the location of the particle at time $n$, we are interested
also in the first hitting times of certain regions in space. For any set $A\subset\Z$, define:
\[
 T_A := \min\{n\geq 0 : X_n\in A\}.
\]
To simplify the notations, for one-point sets we write $T_a:=T_{\{a\}}$. In the case where $a$ is not an integer, the notation $T_a$ will correspond to $T_{\lfloor a \rfloor}$.

In this paper we investigate the following types of unusual behaviour of the random walk:
\begin{itemize}
 \item \emph{slowdown}, which means that at time~$n$ the particle is around~$n^{\nu_0}$,
$\nu_0<1 \wedge \kappa$, so that the particle goes to the right much slower than it typically does;
 \item \emph{backtracking}, that is, at time~$n$ the particle is found around~$(-n^\nu)$, thus performing
an unlikely excursion to the left instead of going to the right (this is, of course, only for RWRE
without reflection);
 \item \emph{speedup}, which means that the particle is going to the right faster than it should
(but still with sublinear speed): at time~$n$ the particle is around~$n^{\nu_1}$, $\kappa<\nu_1<1$ (this is possible only for $\kappa <1$).
\end{itemize}

We refer to all of the above as \emph{moderate deviations},
even for the slowdown
in the ballistic case $\kappa>1$. Indeed, in the latter case the deviation from the typical
position is linear in time, but we have that the large deviation rate function $I$ satisfies $I(0) = 0$,
and the known large deviation results only tell us that slowdown probabilities decay slower than exponentially in $n$ (see, for instance, \cite{CGZ}).

We mention here that in the literature one can find some results on moderate deviations 
for the case of recurrent RWRE (often referred to as RWRE in
``Sinai's regime''), see~\cite{CP,CP04}, and also~\cite{HS} for the continuous space 
and time version.

Now, we state the results we are going to prove in this paper. In addition to (\ref{def_kappa}), we will use the following weak integrability hypothesis:
\begin{equation}
\label{integ_hyp}
\text{there exists $\epsilon_0>0$ such that }  \EE[\rho_0^{-\epsilon_0}]<\infty.
\end{equation}

First, we discuss the results about quenched slowdown probabilities. It turns out that the quenched 
slowdown probabilities behave differently depending on whether one considers RWRE with or without
reflection at the origin. Also, it matters which of the following two events is considered: (i)
the position
of the particle at time~$n$ is at most $n^\nu$, $\nu<\kappa$ (i.e., the event $\{X_n<n^\nu\}$),
or (ii) the hitting time of~$n^\nu$ is greater than~$n$ (i.e., the event $\{T_{n^\nu}>n\}$). 
Here we prove that in all these cases the quenched probability of slowdown is roughly $e^{-n^\beta}$,
where $\beta=1-\frac{\nu}{\kappa}$ for the ``hitting time slowdown'' in the reflected case,
and $\beta=(1-\frac{\nu}{\kappa})\wedge\frac{\kappa}{\kappa+1}$ in the other cases.
More precisely, we have
\begin{theorem}
\label{t_q_slow}
{\bf Slowdown, quenched} Suppose that~(\ref{logrho}),~(\ref{def_kappa}) and~(\ref{integ_hyp}) hold. For $\nu \in (0, 1\wedge \kappa)$ 
the quenched slowdown probabilities behave in the following way. For the reflected RWRE,
\begin{align}
 \lim_{n\to\infty} \frac{\ln(-\ln \Pto[T_{n^{\nu}}>n])}{\ln n} &= 1-\frac{\nu}{\kappa},
\qquad \text{$\PP$-a.s.},\label{refl_hit_q_sl}\\
 \lim_{n\to\infty} \frac{\ln(-\ln \Pto[X_n<n^{\nu}])}{\ln n}
   &= \Bigl(1-\frac{\nu}{\kappa}\Bigr) \wedge \frac{\kappa}{\kappa+1},
\qquad \text{$\PP$-a.s.} \label{refl_loc_q_sl}
\end{align}
For the RWRE without reflection, we obtain
\begin{align}
 \lim_{n\to\infty} \frac{\ln(-\ln \Po[T_{n^{\nu}}>n])}{\ln n}
 &= \Bigl(1-\frac{\nu}{\kappa}\Bigr) \wedge \frac{\kappa}{\kappa+1},
\qquad \text{$\PP$-a.s.},\label{nonrefl_hit_q_sl}\\
 \lim_{n\to\infty} \frac{\ln(-\ln \Po[X_n<n^{\nu}])}{\ln n}
 &= \Bigl(1-\frac{\nu}{\kappa}\Bigr) \wedge \frac{\kappa}{\kappa+1},
\qquad \text{$\PP$-a.s.} \label{nonrefl_loc_q_sl}
\end{align}
\end{theorem}
For a heuristical explanation of the reason for the different behaviours of
the quenched slowdown probabilities we refer to the beginning of Section~\ref{s_q_slowdown}.

For the annealed slowdown probabilities, we obtain that there is no difference between reflecting/nonreflecting cases (at least on the level of precision we
are working here) and also it
does not matter which one of the slowdown events $\{T_{n^{\nu}}>n\}$, $\{X_n<n^{\nu}\}$
one considers. In all these cases, the annealed probability of slowdown decays polynomially,
roughly as $n^{-(\kappa-\nu)}$:
\begin{theorem}
\label{t_a_slow}
{\bf Slowdown, annealed} Suppose that~(\ref{logrho}),~(\ref{def_kappa}) and~(\ref{integ_hyp}) hold. For $\nu \in (0, 1\wedge \kappa)$,
\begin{equation}
\label{a_slowdown}
 \lim_{n\to\infty} \frac{\ln \PA[X_n<n^{\nu}]}{\ln n}
       =\lim_{n\to\infty} \frac{\ln \PA[T_{n^{\nu}}>n]}{\ln n} = -(\kappa-\nu).
\end{equation}
The same result holds if one changes~$\PA$ to~$\PAt$ in~(\ref{a_slowdown}).
\end{theorem}

In the case of RWRE on~$\Z$ (i.e., without reflection at the origin) there
is another kind of untypically slow escape to the right. Namely, before going to $+\infty$,
the particle can make an untypically big excursion to the left of the origin.
While it is easy to control the distribution of the leftmost site touched by this
excursion (e.g., by means of the formula~(\ref{exit_probs}) below), it is
interesting to study the probability that at time~$n$ the particle is far away
to the left of the origin:
\begin{theorem}
\label{t_backtrack}
{\bf Backtracking} Suppose that~(\ref{logrho}),~(\ref{def_kappa}) and~(\ref{integ_hyp}) hold. For $\nu \in (0, 1)$, we have
\begin{align}
 \lim_{n\to\infty} \frac{\ln(-\ln \Po[X_n<-n^{\nu}])}{\ln n} &= \nu \vee \frac{\kappa}{\kappa+1},
 \qquad \text{$\PP$-a.s.}
  \label{q_backtrack}\\
 \lim_{n\to\infty} \frac{\ln(-\ln \PA[X_n<-n^{\nu}])}{\ln n} &= \nu,  
  \label{a_backtrack}
\end{align}
and
\begin{equation}
  \label{backtrack_T}
 \lim_{n\to\infty} \frac{\ln(-\ln \PA[T_{-n^{\nu}}<n])}{\ln n} 
   = \lim_{n\to\infty} \frac{\ln(-\ln \Po[T_{-n^{\nu}}<n])}{\ln n} = \nu
\qquad \text{$\PP$-a.s.}
\end{equation}
\end{theorem}

Another kind of deviation from the typical behaviour is the speedup of the particle,
i.e., at time~$n$ the particle is at a distance larger than~$n^\kappa$ from
the origin (here we of course assume that $\kappa<1$). There are results in the literature that cover the 
\emph{large deviations} case, i.e., the case when at time~$n$ the particle is at distance~$O(n)$
from the origin, see e.g.\ Section~2.3 of~\cite{Zeitouni}, or~\cite{CGZ}.
In this paper we are interested in the probabilities of moderate speedup:
the displacement of the particle is sublinear, but still bigger than in the typical case.
Namely, we show that the quenched probability that~$X_n$ is of order~$n^\nu$,
$\kappa<\nu<1$, is roughly $e^{-n^\beta}$, where $\beta=\frac{\nu-\kappa}{1-\kappa}$.
It is remarkable that the annealed probability is roughly of the same order.
More precisely, we are able to prove the following result:
\begin{theorem}
\label{t_speedup}
{\bf Speedup} Suppose that~(\ref{logrho}),~(\ref{def_kappa}) and~(\ref{integ_hyp}) hold. For $\nu \in (\kappa,1)$ we can control the probabilities of the moderate speedup in the following way:
\begin{equation}
\label{eq_q_speed}
 \lim_{n\to\infty} \frac{\ln(-\ln \Po[X_n> n^{\nu}])}{\ln n} 
   = \lim_{n\to\infty} \frac{\ln(-\ln \Po[T_{n^{\nu}}< n])}{\ln n} 
   = \frac{\nu-\kappa}{1-\kappa},\qquad \text{$\PP$-a.s.},
\end{equation}
and
\begin{equation}
\label{eq_a_speed}
 \lim_{n\to\infty} \frac{\ln(-\ln \PA[X_n > n^{\nu}])}{\ln n} 
  = \lim_{n\to\infty} \frac{\ln(-\ln \PA[T_{n^{\nu}} < n])}{\ln n} = \frac{\nu-\kappa}{1-\kappa}.
\end{equation}
The same result holds for the RWRE with reflection at the origin.
\end{theorem}

For the case $\kappa\in (0,1)$,
the quenched moderate deviations for the random walk on~$\Z$ are well 
summed up by the plot of the following function on Figure~\ref{f_exponent}: 
\[
f(\nu)=\left\{
\begin{array}{ll}
\lim_{n\to\infty} \ln(-\ln \Po[X_n<-n^{-\nu}])/\ln n, & \text{ if } \nu\in (-1,0],\\ 
\lim_{n\to\infty} \ln(-\ln \Po[X_n<n^{\nu}])/\ln n,  & \text{ if } \nu\in (0,\kappa), \\ 
\lim_{n\to\infty} \ln(-\ln \Po[X_n>n^{\nu}])/\ln n,  & \text{ if } \nu\in [\kappa, 1). \\ 
\end{array}
\right.
\]
\begin{figure}[htb]
\centering
\includegraphics[width=12cm]{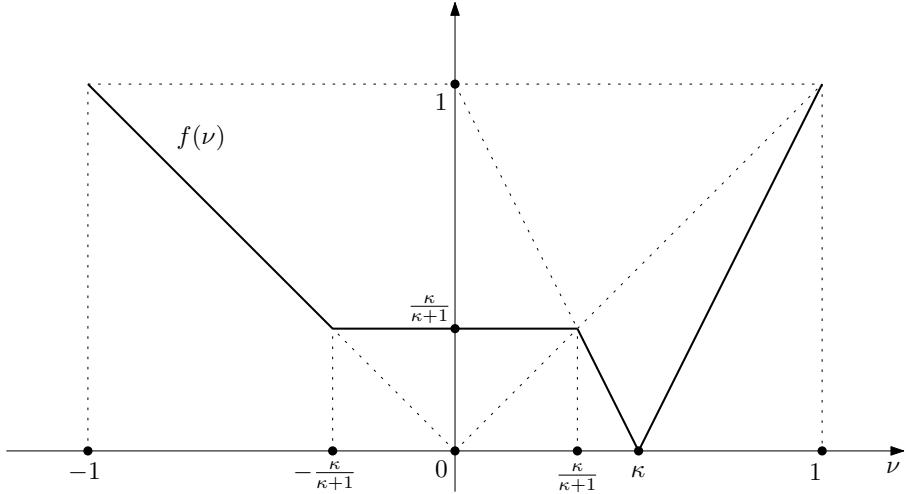}
\caption{The plot of $f(\nu)$, $-1<\nu< 1$}
\label{f_exponent}
\end{figure}

The rest of this paper is organized in the following way. In Section~\ref{s_notations} 
we give the (standard) definition of the potential and the reversible measure for the RWRE.
We then decompose the environment into a sequence of valleys. In this decomposition the valleys 
do not only depend on the environment but the construction is time-dependent.
Also, we derive some basic facts about the valleys needed later. In Section~\ref{s_estimates_env}
we mainly study the properties of that sequence of valleys. In Section~\ref{s_confinement},
we recall some results concerning the spectral properties of RWRE restricted to a
finite interval, and then obtain some bounds on the probability of confinement
in a valley. In Section~\ref{s_induced_rw} we define the induced random walk
whose state is the current valley (more precisely, the last visited 
boundary between two neighbouring valleys) where the particle is located.
Theorems~\ref{t_q_slow}, \ref{t_a_slow}, \ref{t_backtrack}, \ref{t_speedup}
are proved in Sections~\ref{s_q_slowdown}, \ref{s_a_sl}, \ref{s_backtracking},
\ref{s_speedup} respectively. We denote by $\gamma, {\gamma}_0, {\gamma}_1, {\gamma}_2, {\gamma}_3,\ldots$ the ``important'' constants (those that can be used far away from the place where they appear for the first time), and by $C_1,C_2,C_3,\ldots$ the ``local'' 
ones (those that are used only in a small neighbourhood of the place where they appear for the first time), restarting the numeration at the beginning of each section in the latter case. 
All these constants are either universal or depend only on the law
of the environment.

\section{More notations and some basic facts}
\label{s_notations}
An important ingredient of our proofs is the analysis of the {\sl potential} associated with the environment, which was
introduced by Sinai in \cite{Sinai}. The potential, 
denoted by $V= (V(x), \; x\in \Z),$ is a function of
the environment~$\omega$. It is defined in the following way:
\[
V(x) :=\left\{\begin{array}{ll} 
          \sum_{i=1}^x \ln \rho_i, & \text{if } x \geq 1, \\
 0\vphantom{\sum^N}, &  \text{if }  x=0, \\
-\sum_{i=x+1}^0 \ln \rho_i, &\text{if } x\leq -1,
\end{array}\right.
\]
so it is a random walk with negative drift, because ${\bf E}[\ln \rho_0]<0$. This notation is extended on $\R$ by $V(x):=V(\lfloor x \rfloor)$. We also define a reversible measure
\begin{equation}\label{pidef}
\pi(x) :=e^{-V(x)}+e^{-V(x-1)},\qquad \text{for $x\in \Z$,}
\end{equation}
(one easily verifies that $\omega_x\pi(x)=(1-\omega_{x+1})\pi(x+1)$ for all~$x$). 
We will also use the notation
$\pi([x,y])=\sum_{i=\lfloor x \rfloor  -1}^{\lfloor  y \rfloor } \pi (i)$, for $x<y$ two real numbers.

The function~$V(\cdot)$ enables us to define the valleys, parts of the environment which acts as traps for the random walk. The valleys are 
responsible for the sub-ballistic behaviour and hence play a central role for slowdown and speedup phenomena. 

We define by induction the following environment dependent sequence $(K_i(n))_{i \geq 0}$ by
\begin{align*}
K_0(n) =& -n,\\ 
K_{i+1}(n) =& \min \Bigl\{ j \geq K_i(n):  V(K_i(n))-\min_{k\in[K_i(n),j]}V(k)\geq \frac 3{1\wedge \kappa}\ln n,\\  
& \qquad \qquad \qquad \qquad \qquad \qquad \qquad \qquad \qquad \qquad V(j)=\max_{k\geq j} V(k) \Bigr\}.
\end{align*}
The dependence with respect to $n$ will be frequently omitted to ease the notations. The portion of the environment $[K_i,K_{i+1})$ is called the $i$-th valley, and we will prove that for $n$ large enough the valleys are descending 
in the sense that $V(K_{i+1})<V(K_i)$ for all $i\in[0,n]$. We associate to the $i$-th valley the bottom point 
\[
b_i=\inf\Bigl\{x \in [K_{i},K_{i+1}): V(x)= \min_{y\in [K_{i},K_{i+1})} V(y)\Bigr\},
\]
and the depth 
\begin{align*}
H_i&=\max_{x\in [K_{i},K_{i+1})}\Bigl(\max_{y\in [x,K_{i+1})} V(y)-\min_{y\in [K_{i},x)} V(y)\Bigr)\\
 & = \max_{K_i(n)\leq j<k<K_{i+1}(n)} \bigl(V(k)-V(j)\bigr),
\end{align*}
see Figure~\ref{f_valleys}.
\begin{figure}
\centering
\includegraphics[width=12cm]{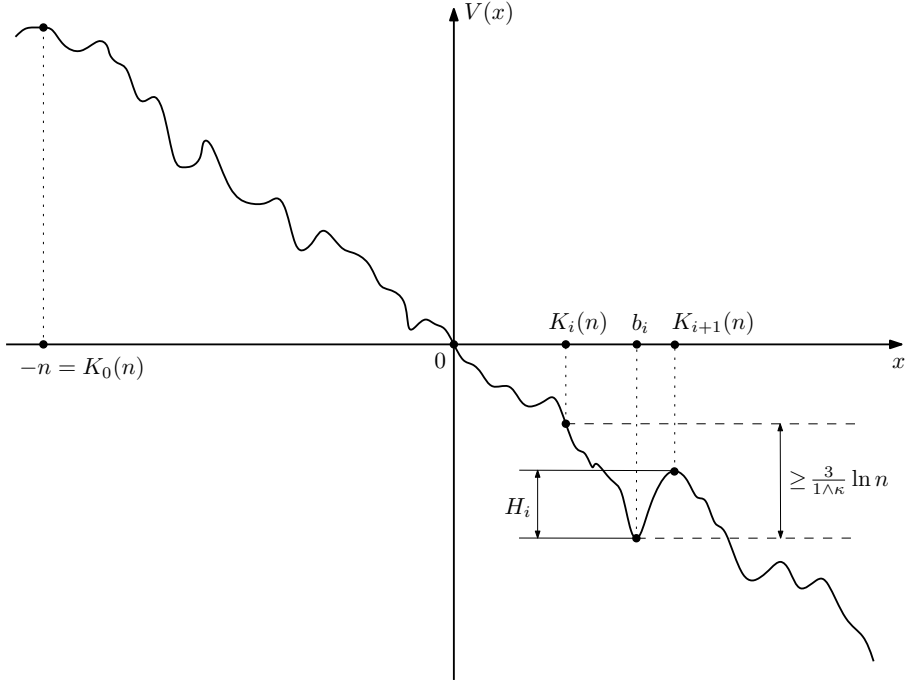}
\caption{On the definition of the sequence of valleys}
\label{f_valleys}
\end{figure}

Let us denote 
\begin{equation}
\label{Ndef}
N_n(m,m')=\{ i\geq 1: \ [K_{i},K_{i+1})\cap [\lfloor m\rfloor, \lfloor m'\rfloor) \neq \emptyset\}
\end{equation}
 and again we will often omit the index $n$. Let us emphasize that we do not include the valley of index 0, which is different from the others because of border issues.

The valleys for $i\geq 1$ are non-overlapping parts of~$\Z$, for any value of~$n$. Moreover the potential in the valleys are i.i.d.~up to space-shift, 
in the sense that for any~$n$ and~$i\geq 1$ 
the sequence of vectors of random length
$ \bigl(V(j)-V(K_{i+1}(n)),j=K_i(n),\ldots,K_{i+1}(n)-1\bigr)$,
$i\geq 1$,
is i.i.d.

We introduce the two following indices which will be used regularly
\begin{equation}
\label{indices}
i_0=\card N(-n,0) \ \text{ and }\  i_1=\card N(-n,n^{\nu}).
\end{equation}

To carry over the proofs easily to the reflected case, we introduce the following notation
\begin{equation}
\label{tilddef}
\widetilde{K}_{i_0}=0 \ \text{ and }\ \widetilde{K}_i=K_i \text{ for $i\geq i_0$}.
\end{equation}

We can estimate the depth of the valleys using a result of renewal theory which concerns the maximum of random walks with negative drift. We refer to~\cite{Feller} for a detailed introduction to renewal theory. Denoting $S=\max_{i\geq 0} V(i)$, under assumptions 
(\ref{logrho}), (\ref{def_kappa}) and~$(iii)$ of Theorem~\ref{kks}, we have
\begin{equation}
\label{fellerthm_nlattice}
\PP[S>h] \sim C_F \, e^{-\kappa h},\qquad \qquad h\to \infty,
\end{equation}
which is a result due to Feller which can be found in this form in~\cite{igle}.

If $(iii)$ in Theorem~\ref{kks} fails, $\ln \rho_0$ is concentrated on $\lambda \Z$ for some~$\lambda>0$, so that $V(\cdot)$ is a Markov chain with i.i.d.~increments of law $\ln \rho_i$. In this case, under our assumptions~(\ref{logrho}) and~(\ref{def_kappa}) we can use a result in~\cite{Spitzer} (p.~218) stating the discrete version of the previous equation. In the case of an aperiodic Markov chain we have
\begin{equation}
\label{fellerthm_lattice}
\PP[S\geq n \lambda] \sim C_F' \, e^{-\kappa \lambda n}, \qquad \qquad n\to \infty,
\end{equation}
and in the general case we obtain similar asymptotics by noticing that $(V(nd+k))_{n\geq 0}$ is aperiodic for $k\in [0,d-1]$ and $d$ the period of $V(\cdot)$ (which is well defined and finite by $(i)$ and $(ii)$).

Hence we can easily deduce from the assumptions~(\ref{logrho}) and~(\ref{def_kappa}) and  equations~(\ref{fellerthm_nlattice}) 
and~(\ref{fellerthm_lattice}) that
\begin{equation}
\label{fellerthm}
\PP[S>h]=\Theta(e^{-\kappa h}),
\end{equation}
where $f(n)=\Theta(g(n))$ means that $f(n)=O(g(n))$ and $g(n)=O(f(n))$.

Let us recall also the following basic fact.
For any integers $a<x<b$, the (quenched) probability for RWRE to reach~$b$ 
before~$a$ starting from~$x$ can be easily computed:
\begin{equation}
\label{exit_probs}
\Po^x[T_b < T_a] = \frac{ \sum_{y=a}^{x-1}
  e^{V(y)}}{ \sum_{y=a}^{b-1}  e^{V(y)}},
\end{equation}
see e.g.\ Lemma~1 in~\cite{Sinai} or formula~(2.1.4) in~\cite{Zeitouni}.


\section{Estimates on the environment}
\label{s_estimates_env}
Let us introduce the event
\begin{equation}
\label{Adef}
A(n)= \Bigl\{ \max_{i\leq 2n} (K_{i+1}-K_i) \leq (\ln n)^2\Bigr\}.
\end{equation}

The following lemma shows that the valleys are not very wide.
\begin{lemma}
\label{widthvalley}
We have
\[
\PP[A(n)^c]=O\Bigl(\frac 1 {n^2}\Bigr).
\]
\end{lemma}

\begin{proof}
We have
\begin{align}
 \PP[A(n)^c] &= \PP\Bigl[\max_{i\leq 2n} (K_{i+1}-K_i) > (\ln n)^2\Bigr]\nonumber\\
 &\leq 2n \PP[K_2-K_1 > (\ln n)^2]+\PP[\overline K_1 > (\ln n)^2],\label{Astep1}
\end{align}
where
\[
\overline K_1(n) = \min \Bigl\{ j \geq 0: -\min_{k\in[0,j]}V(k)\geq \frac 3{1\wedge \kappa}\ln n, \quad V(j)=\max_{k\geq j} V(k) \Bigr\}.
\]
Now 
\begin{align*}
\PP[K_2-K_1 > (\ln n)^2]&=\PP[\overline K_1> (\ln n)^2\mid \max_{i\geq 0} V(i) \leq 0]\\
&\leq \frac{\PP[\overline K_1 > (\ln n)^2]}{\PP[\max_{i\geq 0} V(i)\leq 0]},
\end{align*}
where $\PP[\max_{i\geq 0} V(i)\leq 0]>0$ since $\EE[\ln \rho_0]<0$.
Choose $\ell$ such that $\varepsilon_0\ell > 3(1\wedge \kappa)$, with $\varepsilon_0$ from (\ref{integ_hyp}).
Note that if $V((\ln n)^2) \leq -\frac {3+3\ell}{1\wedge\kappa} \ln n$, $ \min_{j\leq (\ln n)^2} \left(V(j)-V(j-1)\right)\geq - \frac{\ell}{1\wedge\kappa} \ln n$ and $\max_{j\geq (\ln n)^2} V(j)-V((\ln n)^2)\leq \frac 3{1\wedge\kappa} \ln n$, then the set 
\[
 \left\{i \in [0, (\ln n)^2],~V(i)\in \left(-\frac 3{1\wedge\kappa} \ln n,-\frac{3+2\ell}{1\wedge\kappa} \ln n\right)\right\}
\]
 is non-empty. Moreover its largest element~$m$ is such that $\max_{j\geq m} V(j)=V(m)$, 
hence we have $\overline K_1 \leq (\ln n)^2$. This yields
\begin{align}
\label{Astep2}
\PP[\overline K_1 > (\ln n)^2] &   \leq \PP\Bigl[V((\ln n)^2) >-\frac {3+3\ell}{1\wedge\kappa} \ln n 
\\ & \qquad \nonumber \qquad \text{ or } \min_{j\leq (\ln n)^2} \left(V(j)-V(j-1)\right)<  -\frac{\ell}{1\wedge\kappa} \ln n
\\ & \qquad \nonumber \qquad \text{ or } \max_{j\geq (\ln n)^2} V(j)-V((\ln n)^2)> \frac 3{1\wedge\kappa} \ln n\Bigr].
\end{align}

Using (\ref{fellerthm}), we obtain
\begin{equation}
\label{Astep3}
 \PP\Bigl[\max_{j\geq (\ln n)^2} V(j)-V((\ln n)^2)> \frac 3{1\wedge\kappa} \ln n\Bigr] =O(n^{-3}).
\end{equation}

Furthermore, using Chebyshev's inequality and (\ref{integ_hyp}) we get 
\begin{align}
\lefteqn{\PP\Bigl[ \min_{j\leq (\ln n)^2}\left( V(j)-V(j-1)\right)< - \frac{\ell}{1\wedge\kappa} \ln n\Bigr]} \nonumber\\
&\leq (\ln n)^2 \PP\left[\ln \rho_0 < 
-\frac{\ell}{1\wedge\kappa} \ln n\right] \nonumber\\
 &\leq (\ln n)^2 \PP\left[\rho_0^{-\varepsilon_0} > \exp\left(\varepsilon_0\frac{\ell}{1\wedge\kappa}\ln n\right) \right]\nonumber\\
& \leq (\ln n)^2 \EE[\rho_0^{-\epsilon_0}] n^{-\varepsilon_0\ell/(1\wedge \kappa)}\nonumber\\ 
&= o(n^{-3})\, .
\label{Astep5}
\end{align}

Now, since $V(\cdot)$ is a sum of i.i.d.~random variables with exponential moments by the assumptions~(\ref{def_kappa}) and~(\ref{integ_hyp}), we can use large deviations techniques to get
\begin{align}
\PP[V((\ln n)^2) >-C_1 \ln n] &\leq \PP\bigl[\abs{V((\ln n)^2)-\EE[V(1)] (\ln n)^2} >   C_2(\ln n)^2\bigr] \label{Astep4} \\ 
&\leq \exp(- C_3(\ln n)^2)\nonumber\\
 &= o(n^{-3}),\nonumber
\end{align}
since $\EE[V(1)]=\EE[\ln \rho_0]\in(-\infty,0)$. Putting together (\ref{Astep1}), (\ref{Astep2}), (\ref{Astep3}), (\ref{Astep5}) and~(\ref{Astep4}) 
we obtain the result.
\end{proof}

Consider $a\in [0, \nu)$, and define the event
\[
B(n,\nu,a)^c=\Bigl\{ \card\Bigl\{i \in N_n(-n^{\nu},n^{\nu}): H_i \geq \frac a {\kappa} \ln n + \ln \ln n\Bigr\} \geq n^{\nu -a} \Bigr\}.
\]

The following lemma will tell us that asymptotically, between levels $-n^{\nu}$ and $n^{\nu}$ there are at most $n^{\nu-a}$ valleys of depth greater than $(a/ \kappa) \ln n  + \ln \ln n$. 
\begin{lemma}
\label{depthvalley1}
For any $a\in [0,\nu)$, we have
\[
\PP[B(n,\nu,a)^c]=O(n^{-2}).
\]
\end{lemma}

\begin{proof}
 We have easily that (``$\prec$'' means ``stochastically dominated'')
\begin{align*}
 \lefteqn{\card\Bigl\{i \leq N_n(-n^{\nu},n^{\nu}): H_i \geq \frac a {\kappa} \ln n + \ln \ln n\Bigr\}}\\
& \prec \text{Bin}\Bigl(2\lfloor n^{\nu}\rfloor+2,\PP\Bigl[S\geq \frac a {\kappa} \ln n + \ln \ln n\Bigr]\Bigr),
\end{align*}
since we have at most $2\lfloor n^{\nu}\rfloor+2$ integers on the right of which we need an increase of potential of $(a/\kappa)\ln n+\ln \ln n$ to create a valley of sufficient depth.

Using~(\ref{fellerthm}), we have
\[
\PP\Bigl[S\geq \frac a {\kappa} \ln n + \ln \ln n\Bigr]=O\Bigl( \frac{n^{-a}}{(\ln n)^{\kappa}}\Bigr).
\]
Now, using Chebyshev's exponential inequality, we can write
\begin{align*}
\lefteqn{\PP\Bigl[ \text{Bin}\Bigl(2 \lfloor n^{\nu}\rfloor+2,\PP\Bigl[S\geq \frac a {\kappa} \ln n + \ln \ln n\Bigr]\Bigr) 
\geq n^{\nu -a}\Bigr]}\\
 &\leq C_4 \exp(-n^{\nu -a})\exp(C_5 n^{\nu-a}(\ln n)^{-\kappa}),
\end{align*}
and, since $\nu>a$, the result follows.
\end{proof}

We introduce for $m\in \Z^+$ the following event, which, by Lemma~\ref{depthvalley1}, has probability converging to 1,
\begin{equation}
\label{def_b_prime}
B'(n,\nu,m)=\bigcap_{k=1}^{m-1}B(n,\nu, k \nu/m).
\end{equation}

Also, set
\begin{align*}
G(n)^c=&\Bigl\{\max_{ k \geq n} (V(k)-V(n)) \geq \frac 1 {\kappa}(\ln n + 2 \ln \ln n) \Bigr\}
\nonumber\\
    &\bigcup \Bigl\{\max_{ k \geq -n} (V(k)-V(-n)) \geq \frac 1 {\kappa}(\ln n + 2 \ln \ln n) \Bigr\}.
\end{align*}

\begin{lemma}
\label{depthdiff}
We have
\[
\PP[G(n)^c] =O \Bigl(\frac 1 {n (\ln n)^2}\Bigr).
\]
\end{lemma}
\begin{proof}
This is a direct consequence of (\ref{fellerthm}).
\end{proof}
We now show that Lemma \ref{depthdiff} implies that asymptotically, in the interval $[-n,n]$, the deepest valley we can find
 has depth lower than $ \frac 1 {\kappa}(\ln n + 2 \ln \ln n)$. Let
\begin{equation}
\label{Gdef}
G_1(n)=\Bigl\{\ \max_{ i\in [-n, n]} \max_{k\geq i} ( V(k)-V(i) ) \leq \frac {1} {\kappa}(\ln n+ 2\ln \ln n) \Bigr\}.
\end{equation}

\begin{lemma}
\label{depthvalley}
For $\PP$-almost all $\omega$, there is $N= N(\omega)$ such that $\omega \in G_1(n)$ for $n \geq N$.
\end{lemma}

\begin{proof}
By symmetry, it suffices to give the proof for
\begin{equation}
G_2(n)=\Bigl\{\ \max_{ i\in[0,n]} \max_{k\geq i} ( V(k)-V(i) ) \leq \frac {1} {\kappa}(\ln n+ 2\ln \ln n) \Bigr\}
\end{equation}
instead of $G_1(n)$.
Let 
\[
n_0 : = \min\Bigl\{j \geq 0:  \max_{k\geq i} ( V(k)-V(i) ) \leq  \frac {1} {\kappa}(\ln i + 2\ln \ln i), \, \forall i\geq j \Bigr\}
\]
and
\[
K = \max_{0 \leq i\leq n_0}\max_{k\geq i} ( V(k)-V(i) ) .
\]
Due to Lemma \ref{depthdiff}, $n_0$ is finite $\PP$-almost surely. Now, take $N$ large enough such that $N \geq n_0$ and
\[
\frac {1} {\kappa}(\ln N +2\ln \ln N)\geq K.
\]
Then for $n\geq N$, let $\ell\in[0,n]$ be such that $ \max_{ i\in[0,n]} \max_{k\geq i} ( V(k)-V(i) ) = \max_{k\geq \ell} ( V(k)-V(\ell) )$. We have either $\ell \leq n_0$ and then $\max_{k\geq \ell} ( V(k)-V(\ell) ) \leq K$ by the definition of $K$, or $\ell > n_0$ and then, by the definition of $n_0$,
$\max_{k\geq \ell} ( V(k)-V(\ell) ) \leq \frac {1} {\kappa}(\ln \ell + 2\ln \ln \ell)\leq \frac {1} {\kappa}(\ln n + 2\ln \ln n)$.
\end{proof}
Let us define
\begin{align*}
D(n)^c=&\Bigl\{\ \max_{ i\in[0,n]} \max_{k\geq i} ( V(k)-V(i) ) \leq \frac {1} {\kappa}(\ln n-4\ln \ln n) \Bigr\}
\nonumber\\
    &\bigcup \Bigl\{\ \max_{ i\in[-n,0]} \max_{k\geq i} ( V(k)-V(i) ) \leq \frac {1} {\kappa}(\ln n-4\ln \ln n) \Bigr\}.
\end{align*}
\begin{lemma}
\label{depthvalley2}
We have
\[
\PP[D(n)^c]=O(n^{-2}).
\]
\end{lemma}

\begin{proof}

First, we notice that
\begin{align*}
\PP[D(n)^c] & \leq 2\PP\Bigl[\max_{ i \in [0,\frac{n}{\lfloor (\ln n)^2\rfloor}]} \max_{k \leq (\ln n)^2} V(i(\ln n)^2+k)-V(i(\ln n)^2) \\
& \qquad\qquad\qquad\qquad\qquad \leq \frac {1} {\kappa}(\ln n-4\ln \ln n)\Bigr]+\PP[A(n)^c],
\end{align*}
where $\PP[A(n)^c]=O(n^{-2})$ by Lemma~\ref{widthvalley}.

Let us introduce
\[
D^{(1)}(n)=\Bigl\{\max_{ k >\lfloor (\ln n)^2 \rfloor} V(k)-V(0) \geq \frac {1} {\kappa}(\ln n-4\ln \ln n)\Bigr\},
\]
then we have
\begin{align*}
\PP[D^{(1)}(n)]\leq &\PP\Bigl[\max_{ k \geq 0} V(k)-V(0) >  \frac {1} {\kappa}(\ln n-4\ln \ln n)\Bigr]
\nonumber\\
    &+\PP\Bigl[\max_{ k \geq 0} V(k)-V(0) \neq \max_{ k \leq (\ln n)^2} V(k)-V(0)\Bigr] =\Theta\Bigl(\frac {(\ln n)^4} {n}\Bigr),
\end{align*}
using a reasoning similar to the proof of Lemma~\ref{widthvalley} (cf.\ equations~(\ref{Astep2}) and~(\ref{Astep3})) to show that the second term is at most $O(n^{-2})$.

So, we obtain for~$n$ large enough
\[
\PP[D(n)^c] \leq 2\Bigl(1-\frac{C_6(\ln n)^4} {n}\Bigr)^{n/(\ln n)^2}\leq 2\exp\left(-C_7 (\ln n)^2\right),
\]
hence the result.
\end{proof}

Finally, let us introduce
\[
F(n)=\bigl\{ \min_{i \in [-n,n]} (1-\omega_i) >n^{-3/\epsilon_0}\bigr\}.
\]

\begin{lemma}
\label{easy_backtrack}
We have
\[
\PP[F(n)^c]=O\Bigl(\frac 1 {n^2}\Bigr).
\]
\end{lemma}

\begin{proof}
We notice that $1-\omega_i\geq \min(1/2, \rho_i/2)$, so that it is enough to prove that $\PP[\rho_i<2n^{-3/\epsilon_0}]=O(n^{-3})$ which is a consequence of~(\ref{integ_hyp}), since by Chebyshev's inequality
\[
\PP\Bigl[\rho_i^{-1}>\frac{n^{3/\epsilon_0}}{2}\Bigr]\leq \frac {2^{\epsilon_0}\EE[ \rho_0^{-\epsilon_0}]}{n^3}.
\]
\end{proof}

Using the Borel-Cantelli Lemma one can obtain that for $\PP$-almost all~$\omega$ and~$n$ large enough, we have $\omega \in A(n)\cap B'(n,\nu,m) \cap G_1(n) \cap D(n)\cap F(n)$. That is, the width of the valleys is lower than $(\ln n)^2$, their depth lower than $(\ln n + 2 \ln \ln n)/\kappa$, we can control the number of valleys deeper than
$\frac {a} {\kappa}\ln n-\ln \ln n$, and there is at least one valley of depth $(\ln n -4 \ln \ln n)/\kappa$.

Due to the definition of the valleys, the potential goes down at least by $\frac{3}{1\wedge \kappa}\ln n$ in a valley  and on $G_1(n)$ 
the biggest increase of potential is lower than $\frac 1{\kappa} (\ln n+ 2 \ln \ln n)$ for all valleys in $[-n,n]$. In particular, on $G_1(n)$, $(V(K_i))_{i\leq 2n}$ is a decreasing sequence and we have
\begin{align*}
V(b_{i+1}) & \leq  V(b_i) - \frac{3}{1\wedge \kappa}\ln n + \frac 1{\kappa} (\ln n+ 2 \ln \ln n) 
\nonumber\\
& \leq  V(b_i) - \frac{2}{1\wedge \kappa}\ln n + 
\frac 2{\kappa} \ln \ln n
\end{align*}
implying using~(\ref{pidef}) that for all valleys in $[-n, n]$,
\begin{equation}
\label{weightbottom}
\qquad \pi(b_i)\leq 2e^{-V(b_i)}\leq \frac {2(\ln n)^{2/\kappa}}{n^{2/(1\wedge \kappa)}} \pi(b_{i+1})\leq \frac 12 \pi(b_{i+1}).
\end{equation}

In a similar fashion, we can give an upper bound for $V(K_{i})-V(b_i)$ on $G_1(n)\cap F(n)$. 
We claim that on $G_1(n)\cap F(n)$, for a constant $\gamma_0$,
\begin{equation}
\label{bound_pot}
V(K_{i})-V(K_{i+1}) \leq V(K_{i})-V(b_i)\leq \gamma_0 \ln n.
\end{equation}
To show (\ref{bound_pot}), let~$x$ be the smallest integer larger than $K_{i}$ such that $V(x)\leq V(K_{i})-(3/(1\wedge \kappa)) \ln n$.
 By definition of $K_{i+1}$ it satisfies $V(x)\leq V(K_{i+1})$. But on $F(n)$ we know that $V(x)\geq V(K_{i}) -(3/(1\wedge \kappa)+3/\epsilon_0)\ln n$. Recalling that on $G_1(n)$ we have $V(b_i)\geq V(K_{i+1})-(2/\kappa)\ln n$, we get for $n$ large enough
\begin{align*}
V(K_i)-V(b_i)& \leq V(K_i)-(V(K_{i+1})-\frac{2}{\kappa} \ln n)\nonumber\\
             & \leq V(K_i)-(V(x)-\frac{2}{\kappa} \ln n)\nonumber\\
             & \leq \Bigl(\frac{3}{1 \wedge \kappa} +\frac{3}{\epsilon_0} +\frac{2}{\kappa}\Bigr) \ln n\, .
\end{align*}

\section{Bounds on the probability of confinement}
\label{s_confinement}
In this section, let $I=[a,c]$ be a finite interval of~$\Z$ containing at least four points and let the potential~$V(x)$ 
be an arbitrary function defined for $x \in [a-1, c]$, with $V(a-1)=0$. This potential defines transition probabilities given by $\omega_x = e^{-V(x)}/\pi(x)$, $x \in [a,c]$ where $\pi(x)$ is defined as in (\ref{pidef}) (taking $V(a-1)=0$ is no loss of generality since the transition probabilities remain the same if we replace $V(x)$ by $V(x) + c$, $\forall x$).
We denote by~$X$ the Markov chain restricted on~$I$ in the following
way: the transition probability $\omega_a$ from~$a$ to $a+1$ is defined as above, and with probability~$1-\omega_a$
the walk just stays in~$a$; in the same way, we define the reflection at the other border~$c$. We denote 
\begin{align*}
H_+&=\max_{x\in[a, c]} \Bigl(\max_{y \in [x,c]} V(y) -\min_{y\in [a,x)} V(y)\Bigr) ,\\
H_-&=\max_{x\in[a, c]} \Bigl(\max_{y \in [a,x]} V(y) -\min_{y\in (x,c]} V(y)\Bigr),
\end{align*}
and
\[
H=H_+\wedge H_-.
\]
Let us denote also by
\[
 {\tilde M} = \max_{y \in [a,c]} V(y) -\min_{y\in [a,c]} V(y)
\]
the maximal difference between the values of the potential in the interval $[a,c]$.
Also, we set 
\begin{align*}
f&= \begin{cases} 
                  c, & \text{ if } H=H_+, \\
                  a, & \text{ otherwise. }
    \end{cases}
\end{align*}

To avoid confusion, let us mention that the results of this section
(Propositions~\ref{spectral}, \ref{cost_to_climb}, \ref{Lboundconf}) hold for both
the unrestricted and restricted random walks (as long as the starting point belongs to~$I$).
First, we prove the following
\begin{proposition}
\label{spectral}
There exists ${\gamma}_1>0$, such that for all $u\geq 1$ 
\begin{align*}
\lefteqn{\max_{x\in I} \Po^x\Bigl[\frac{T_{\{a,c\}}}{{\gamma}_1(c-a)^3((c-a)+{\tilde M})e^H}>u \Bigr]}\\
&\leq \max_{x\in I} \Po^x\Bigl[\frac{T_{f}}{{\gamma}_1(c-a)^3((c-a)+{\tilde M})e^H}>u \Bigr] \\ &\leq e^{-u}.
\end{align*}
\end{proposition}

\begin{proof}
The first inequality is trivial, we only need to prove the second one. In the following we will suppose that $H=H_+$ (so that $f=c$), otherwise we can apply the same argument by inverting the space. We denote by~$b$ the leftmost point in the interval $[a,c]$ with minimal potential.

We extend the Markov chain on the interval~$I$ to a Markov chain on the interval $I'=[a,c+1]$ in the following way.
Let $V(c+1):=V(b)$, yielding  $\omega_{c+1}=\bigl(1+e^{-(V(c)-V(b))}\bigr)^{-1}$.  Again, with probability $1- \omega_{c+1}$, the Markov chain goes from $c+1$ to $c$, and with probability $\omega_{c+1}$, the Markov chain just stays in $c+1$.

%

Let us denote by ${\hat X}_t$ the continuous time version of the Markov chain on~$I'$
(i.e., the transition probabilities become transition rates). 
The reason for considering continuous time is the following:
we are going to use spectral gap estimates, and these are better suited for 
continuous time in this context (mainly due to the fact that the discrete-time
random walk is periodic).
We define the probability measure~$\mu$ on~$I'$ which is reversible (and therefore invariant) for~${\hat X}$ in the following way
\[
\mu(x)=\pi(x)\Bigl(\sum_{y\in I'} \pi(y)\Bigr)^{-1},
\]
for all $x\in I'$, where $\pi$ is as in (\ref{pidef}) with the potential defined above, satisfying $V(a-1) =0$ and $V(c+1)=V(b)$. 
Now, the goal is to bound the spectral gap $\lambda(I')$ from below. 
We can do this using a result of~\cite{Miclo}:
\begin{equation}
\label{micloprop}
 \frac 1 {4B^{I'}} \leq \lambda(I') \leq \frac 2 {B^{I'}},
\end{equation}
where $B^{I'}=\min_{i \in I'} ( B_-^{I'}(i) \wedge B_+^{I'}(i))$ and 
\begin{align*}
 B_+^{I'}(i) &= \max_{x>i}\left(\sum_{y=i+1}^x (\mu(y)(1-\omega_y))^{-1}\right)\mu[x,c+1],\quad i \in [a,c]\\
 B_-^{I'}(i) &= \max_{x<i}\left(\sum_{y=x}^{i-1} (\mu(y)\omega_y)^{-1}\right)\mu[a,x], \quad i\in [a+1, c+1]
\end{align*}
and $ B_+^{I'}(c+1)=  B_-^{I'}(a)= 0$.
Obviously, we have $B^{I'} \leq B_-^{I'}(c+1)$. 
Moreover, since~\eqref{pidef} implies that $\omega_x\pi(x)=e^{-V(x)}$ for any $x\in I'$, we can write
\begin{align*}
B_-^{I'}(c+1) &= \max_{x\leq c} \Bigl(\sum_{y=x}^{c} \frac{1}{\omega_y\pi(y)}\Bigr)
           \Bigl(\sum_{y=a}^x \pi(y)\Bigr) \\
&= \max_{x\leq c} \Bigl(\sum_{y=x}^{c} e^{V(y)}\Bigr)
           \Bigl(\sum_{y=a}^x (e^{-V(y)}+e^{-V(y-1)})\Bigr) \\
& \leq 2 \max_{x\leq c}\Bigl(\sum_{y=x}^{c} e^{V(y)}\Bigr)\Bigl(\sum_{y=a}^x e^{-V(y)}\Bigr) \\
 &\leq 2(c-a)^2 e^H.
\end{align*}
This yields
\[
\lambda(I') \geq \frac 1 {8(c-a)^2 e^H}.
\]

Using Corollary 2.1.5 of \cite{Saloff}, we obtain that for $x,y\in I'$ and $s>0$
\[
\abs{\Po^x[{\hat X}_s=y] - \mu(y)} \leq \Bigl(\frac{\mu(y)}{\mu(x)}\Bigr)^{1/2} \exp(-\lambda(I')s).
\]

We apply this formula for $y=c+1$. Note that, using~(\ref{pidef}), we obtain that
$(\mu(c+1)/\mu(x))^{1/2} \leq \sqrt{2}e^{{\tilde M}/2}$ for any $x\in(a,c)$. 
So, for $s:=C_1(c-a)^2((c-a)+{\tilde M})e^H$, if~$C_1>4$ is chosen large enough
\[
\abs{\Po^x[{\hat X}_s=c+1] - \mu(c+1)} \leq \sqrt{2}e^{-C_1(c-a)/8} < \frac {1}{8(c-a)},
\]
and, since $\mu(c+1) \geq 1/2(c+1-a) \geq 1/(4(c-a))$, we obtain
\[
\min_{x\in I'} \Po^x[{\hat X}_s=c+1] \geq \frac 1 {8(c-a)}.
\]

Let us divide $[0,t]$ into $N:=\left\lfloor  t/s \right\rfloor$ subintervals. Using the above inequality and Markov's property we obtain (${\hat T}$ stands for the hitting time with respect
to~${\hat X}$)
\begin{align*}
\Po^x[{\hat T}_{c}>t] &\leq \Po^x[{\hat T}_{c+1}>t]\\
 &\leq \Po^x[{\hat X}_{sk}\neq c+1, k=1,\ldots, N]\\
&\leq\Bigl(1-\frac 1 {8(c-a)}\Bigr)^N \\
 &\leq \exp\Bigl(-\frac{N}{8(c-a)}\Bigr)\\
  &\leq \exp\Bigl(-\frac t{8C_1(c-a)^3((c-a)+{\tilde M})e^H}\Bigr)\exp\Bigl(\frac 1 {8(c-a)}\Bigr).
\end{align*}

The estimates on the continuous time Markov chain transfer to discrete time. Indeed, there exists a family $({\mathbf e}_i)_{i\geq 1}$ of exponential random variables of parameter~$1$, such that the $n$-th jump of the continuous time random walk occurs at $\sum_{i=1}^n {\mathbf e}_i$. These random variables are independent of the environment and the discrete-time random walk. Moreover, 
$P[{\mathbf e}_1+\cdots+ {\mathbf e}_n \geq n]$ 
$\geq 1/3$, for all~$n$. So, for any~$t$,
\[
\frac{1}{3} \PP[ T_{c} \geq t] \leq \PP[ T_{c} \geq t] \PP[{\hat T}_{c}\geq T_{c}] = \PP[ T_{c} \geq t, {\hat T}_{c}\geq T_{c} ]\leq \PP[{\hat T}_{c} \geq t],
\]
 Hence, we have for all $v > 0$
\[
\max_{x\in I} \Po^x\Bigl[\frac{{ T}_{c}}{8(1+v) C_1(c-a)^3((c-a)+{\tilde M})e^H}>u \Bigr] 
\leq \Bigl(3 e^{1/8}e^{-v u}\Bigr)e^{- u},
\]
for all $u\geq 0$. Hence for $u\geq 1$, choosing~$v$ large enough 
in such a way that $3\exp(\frac{1}{8}-v)\leq 1$, we obtain the result with ${\gamma}_1=8C_1(1+v)$.
\end{proof}
Next, we recall the following simple upper bound on hitting probabilities:
\begin{proposition}
\label{cost_to_climb}
There exists $\gamma_2$ such that for any $x,y$ and $h\in[x,y]$ we have
\[
 \Po^x[T_y<s] \leq \gamma_2(1+s)\frac{\pi(h)}{\pi(x)}.
\]
\end{proposition}
\begin{proof}
We can adapt Lemma~3.4 of~\cite{CP} (which used a uniform ellipticity condition). We remain in the continuous time setting and, 
considering the 
event that $y$ is visited before time $s$ and left again at least one time unit later (on which $\int_0^{s+1} \1{\hat X_u=y}du \geq 1$),
we have 
\begin{equation}
\int_0^{s+1} \Po^x[\hat X_u=y]du \geq \Po^x[\hat T_y<s]\cdot P[{\mathbf e}_1 \geq 1]
\end{equation}
where  ${\mathbf e}_1$ is an exponential random variable of parameter~$1$.
Hence
\begin{align*}
\Po^x[\hat T_y<s] & \leq \Po^x[\hat T_{h}<s] \\
             & \leq e \int_0^{s+1} \Po^x[\hat X_u=h]du \\
             & = e \int_0^{s+1} \frac{\pi(h)}{\pi(x)} \Po^{h}[\hat X_u=x]du \\
             & \leq e  (s+1) \frac{\pi(h)}{\pi(x)}.
\end{align*}

Again, one can easily transfer the estimates on the continuous time Markov chain to discrete time.
\end{proof}

Let us now introduce
\begin{align*}
H_+^*&=\max_{x\in[a+1, c-1]} \Bigl(\max_{y \in [x,c-1]} V(y) -\min_{y\in [a+1,x)} V(y)\Bigr) ,\\
H_-^*&=\max_{x\in[a+1, c-1]} \Bigl(\max_{y \in [a+1,x]} V(y) -\min_{y\in (x,c-1]} V(y)\Bigr),
\end{align*}
and
\[
H^*=H_+^*\wedge H_-^*.
\]

We obtain a lower bound on the confinement probability in the following proposition.
Recall that~$b$ is the leftmost point in the interval $[a,c]$ with minimal potential.
\begin{proposition}
\label{Lboundconf}
Suppose that $c-1$ has maximal potential on $[b,c-1]$ and $a$ has maximal potential on $[a,b]$.
Then, there exists ${\gamma}_3>0$, such that for all $u \geq 1$
\[
\min_{x\in I} \Po^x\Bigl[ {\gamma}_3 \ln(2(c-a)) \frac{T_{\{a,c\}}}{e^{H^*}}\geq u\Bigr] \geq \frac 1 {2(c-a)}e^{-u},
\]
if $e^{H^*}\geq 16{\gamma}_2$.
\end{proposition}

\begin{proof}
Noticing that
\[
\min_{b<h<c-1} \frac{\pi(h)}{\pi(b)}\leq 2 e^{-H_+^*} \text{ and }\min_{a+1<h<b} \frac{\pi(h)}{\pi(b)}\leq 2 e^{-H_-^*},
\]
we can apply Proposition~\ref{cost_to_climb} to obtain that 
\begin{equation}
\label{Lboundstep1}
\text{for all }s \geq 1,\qquad \Po^b[T_{\{a,c\}}<s]\leq 8{\gamma}_2 s e^{-H^*},
\end{equation}
Hence for $s=e^{H^*}/(16 \gamma_2)\geq 1$, the right-hand side of the previous inequality equals~$1/2$.

Now, using the exit probability formula~(\ref{exit_probs}), we obtain that
\begin{equation}
\label{Lboundstep2}
\min_{x\in I} \Po^x[T_b<T_{\{a,c\}}] \geq (c-a)^{-1}.
\end{equation}
Denoting $N=\left\lceil t/s \right\rceil$, we obtain for $x\in I$,
\begin{align*}
 \Po^x[T_{\{a,c\}}>t]    & \geq (2(c-a))^{-(N+1)} \nonumber\\
             & \geq \exp\Bigl( - \frac{ C_2 t \ln (2(c-a))}{e^{H^*}} \Bigr)(2(c-a))^{-1}.
\end{align*}
We used the following reasoning in the above calculation.
Start from any $x\in (a,c)$, by~(\ref{Lboundstep2}) the particle hits~$b$
before~$\{a,c\}$ with probability at least $(c-a)^{-1}$. Then, during~$s$
time units, $\{a,c\}$ will not be hit with probability at least~$1/2$.
After that, the particle is found in some $x'\in (a,c)$ and at least~$s$
time units elapsed from the initial moment. So the cost of preventing 
the occurrence of $T_{\{a,c\}}$ during any time interval of length~$s$
is at most $(2(c-a))^{-1}$.
The result follows for ${\gamma}_3$ large enough.
\end{proof}

Our main application of Proposition~\ref{spectral} and Proposition~\ref{Lboundconf}, will be to control the exit times of valleys, more precisely we will be able to give upper bounds on the tail of $T_{\{K_i,K_{i+1}\}}$  and lower bounds on the tail of $T_{\{K_i-1,K_{i+1}+1\}}$ in terms of $H_i$.

\section{Induced random walk}
\label{s_induced_rw}

Let us denote $(s_k(n))_{k\geq 0}$ the sequence defined by
\begin{align*}
s_0(n)&=0,\\
s_{i+1}(n) &=\min \{ j\geq s_i(n):\ X_j \in \{K_l(n),l\geq 0 \} \}.
\end{align*}

Then, we define $Y_i=X_{s_i}$, the embedded random walk with state space $\{K_l,l\geq 0\}$, enumerating the successive valleys we visit and $l_n(\nu)=\max \{i : s_i \leq T_{n^{\nu}}\}$ the numbers of steps made by the embedded random walk to reach $[n^{\nu}; \infty)$. For the reflected case, we will use the same notation, replacing  $\{K_l,l\geq 0\}$ with $\{\widetilde{K}_l,l\geq 0\}$ defined in~(\ref{tilddef}).

Recall~(\ref{indices}) and let us denote
\[
\xi^{\nu}(i)= \card \{ j \in [0, l_n(\nu)]: Y_j=K_{i+1}, Y_{j+1}=K_{i} \} \text{ for } i=i_0+1, \ldots , i_1-1,
\]
and in order to carry over the proofs to the reflected case
\[
\widetilde{\xi}^{\nu}(i)= \card \{ j \in [0, l_n(\nu)]:  Y_j=\widetilde{K}_{i+1}, Y_{j+1}=\widetilde{K}_{i} \} \text{ for } i=i_0+1, \ldots , i_1-1.
\]

 Moreover, we introduce the real time elapsed, i.e.~in the clock of $X_n$, during the first left-right crossing of the $i$-th valley
\[
T^{\text{next}}(i)=T_{K_{i+1}}\circ\theta(\text{next}(i))-\text{next}(i),
\]
where $\theta$ denotes the time-shift for the random walk and
\[
\text{next}(i)=\inf\{n\geq 0: X_n=K_i, T_{K_{i+1}}\circ \theta(n)<T_{K_{i-1}}\circ \theta(n)\}.
\]
  In this way, each time the embedded random walk backtracks, $T^{\text{next}}(i)$ is the time the walk will need to make the necessary left-right crossing of the corresponding valley. Recall~(\ref{Ndef}). Conditionally on $(Y_i)_{i\geq 1}$ we have that (``dir'' stands for ``direct'', and ``back'' stands for ``backtrack'')
\begin{equation}
\label{decompos}
T_{n^{\nu}} = \Tinit+\Tdir+\Tback+\Tleft+\Tright,
\end{equation}
where 
\begin{align*}
 \Tinit &=
 \begin{cases}
  T_{K_{i_0+1}}, &\text{ if $T_{K_{i_0+1}}<T_{K_{i_0}}$, }\\
  T_{K_{i_0}}+T^{\text{next}}(i_0)\circ \theta(T_{K_{i_0}}), &\text{ else,}
 \end{cases}  \\
 \Tleft &=
  \begin{cases}
   \card\{i\leq T_{n^{\nu}}:  X_i < K_1\}\qquad \qquad \qquad \qquad \ \ \ \ \text{ without reflection,}\\  \\
   \sum_{j=0}^{l_n(\nu)} 
\1{Y_j=\widetilde{K}_{i_0+1},Y_{j+1}=\widetilde{K}_{i_0}} &\\ 
         \ \ \times \Bigl(T_{K_i}\circ \theta(s_j)-s_j+T^{\text{next}}(i) \circ (T_{K_i}\circ \theta(s_j))\Bigr)\text{ with reflection,}
 \end{cases}\\
\Tright &=T_{n^{\nu}}\circ \theta(\text{next}^*(i_1))-\text{next}^*(i_1),\\
  \Tdir &=\sum_{i=i_0+1}^{i_1-1} T^{\text{next}}(i)\circ \theta(T_{K_{i}}),\\
 \Tback &=\begin{cases}
\sum_{i=1}^{i_1-1}\sum_{j=0}^{l_n(\nu)} \1{Y_j=K_{i+1},Y_{j+1}=K_i} \\
         \ \ \times \Bigl(T_{K_i}\circ \theta(s_j)-s_j+T^{\text{next}}(i) \circ (T_{K_i}\circ \theta(s_j))\Bigr) \text{ without reflection,}\\    \\
\sum_{i=i_0+1}^{i_1-1}\sum_{j=0}^{l_n(\nu)} \1{Y_j=K_{i+1},Y_{j+1}=K_i} \\
         \ \ \times \Bigl(T_{K_i}\circ \theta(s_j)-s_j+T^{\text{next}}(i) \circ (T_{K_i}\circ \theta(s_j))\Bigr) \text{ with reflection},\\   
\end{cases}
\end{align*}
where $\text{next}^*(i_1)=\inf\{n\geq 0: X_n=K_{i_1}, T_{n^{\nu}}\circ \theta(n)<T_{K_{i_1-1}}\circ \theta(n)\}$. 
In the reflected case, replace~$K_i$ with $\widetilde{K}_i$ in all the above definitions except for that of~$\Tleft$. This decomposition is illustrated on Figure~\ref{f_decompo} for the non-reflected case.

\begin{figure}
\centering
\includegraphics[width=\textwidth]{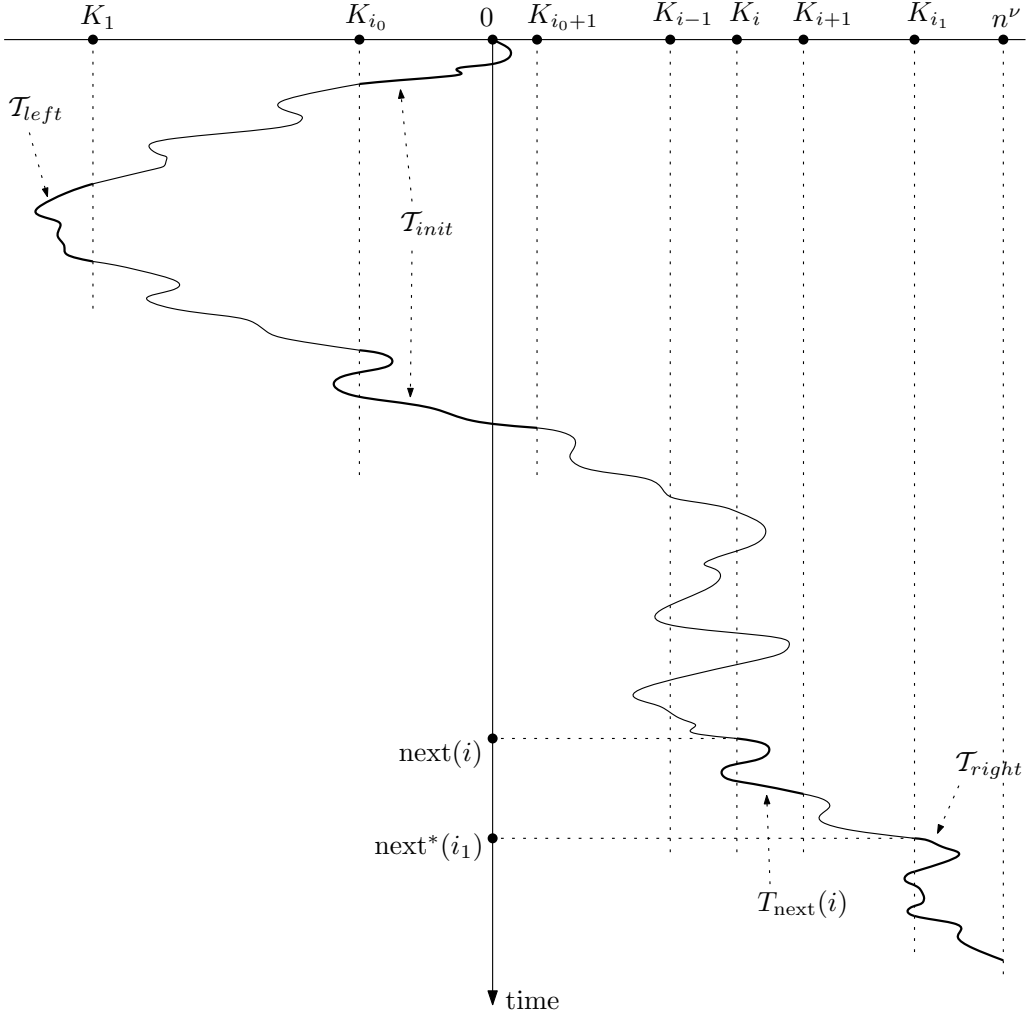}
\caption{On the decomposition~(\ref{decompos}) of $T_{n^{\nu}}$}
\label{f_decompo}
\end{figure}

In the non-reflected case, we have the following equalities in law (for each~$\omega$):
\begin{align}
 \Tinit &=\overline{\tau}(0),\label{defTinit}\\
 \Tright &=\overline{\tau}(n^{\nu}),\label{defTright}\\
 \Tdir &=\sum_{i=i_0+1}^{i_1-1} \tau_+^{(0)}(i),\label{defTdirect}\\
 \Tback &=\sum_{i=1}^{i_1-2} ( \tau_+^{(1)}(i)+\tau_-^{(1)}(i)+ \cdots + \tau_+^{(\xi^{\nu}(i))}(i)+\tau_-^{(\xi^{\nu}(i))}(i)) 
\label{defTbacktrack}\\ \nonumber & \qquad+ \sum\limits_{j=1}^{\xi^\nu(i_1 -1)} \tau_+^{(j)}(i_1-1)+ \tau_-^{\rm{last}, (j)}, 
\end{align}
where $\tau_+^{(j)}(i)$,  $\tau_-^{(j)}(i)$ and $\tau_-^{\rm{last}, (j)}$ are independent sequences of i.i.d.\ random variables described as follows. First, $\tau_+^{(j)}(i)$ is a sequence of independent random variables with the same law as $T_{K_{i+1}}$ under $\Po^{K_i}[~\cdot\mid T_{K_{i+1}}<T_{K_{i-1}}]$. Then, $\tau_-^{(j)}(i)$ is a sequence of independent random variables with the same law as $T_{K_{i}}$ (under $\Po^{K_{i+1}}[~\cdot\mid T_{K_{i}}<T_{K_{i+2}}]$) and $\tau_-^{\rm{last},j}$ is a sequence of independent random variables with the same law as $T_{K_{i_1-1}}$ under $\Po^{K_{i_1}}[~\cdot\mid T_{K_{i_1-1}} < T_{n^{\nu}}]$.
Clearly, the random variable $\overline{\tau}(0)$ (respectively, $\overline{\tau}(n^{\nu})$) has the same law as $T_{K_{i_0+1}}$ (respectively, $T_{n^{\nu}}$) under $\Po[~\cdot\mid T_{K_{i_0+1}}<T_{K_{i_0-1}}]$ (respectively, $\Po^{K_{i_1}}[~\cdot\mid T_{n^{\nu}}<T_{K_{i_1-1}}]$).

In the reflected case, we simply replace $K_i$ by $\widetilde{K}_i$, $\xi^{\nu}(i)$ by $\widetilde{\xi}^{\nu}_i$ and $\omega$ by $\tilde{\omega}$. 

We want to give bounds on the number of backtracks between valleys before the walk reaches $\lfloor n^{\nu}\rfloor$. Denote
\begin{equation}
\label{Bdef}
\Bk(n):=\card\{ i\geq 1 : s_{i+1}(n) \leq T_{n^{\nu}},\ Y_{i+1}<Y_i\}=\sum_{i=1}^{i_1-1} \xi^{\nu}(i).
\end{equation}

By~(\ref{exit_probs}), we obtain that for $i\leq i_1$, $\PP$-a.s.\ for~$n$ large enough,
\begin{align}
\Po^{K_i}[T_{K_{i+1}}>T_{K_{i-1}}]&=\Bigl(\sum_{j=K_{i-1}}^{K_{i+1}-1} e^{V(j)}\Bigr)^{-1}\sum_{j=K_i}^{K_{i+1}-1} e^{V(j)}\label{crossingprob}\\
&\leq \max_{i\leq n} (K_{i}-K_{i-1}) \frac{(\ln n)^{2/\kappa}}{n^{2/(1\wedge\kappa)}}\nonumber\\
&\leq n^{-3/2}, \nonumber
\end{align}
since $\max_{i\leq n} (K_{i+1}-K_{i})\leq (\ln n)^2$ on $A(n)$ and, due to Lemma \ref{depthvalley}, with the same argument as for (\ref{weightbottom}), we have $V(K_{i-1})-V(x) \geq \frac{2}{1\wedge \kappa} \ln n -\frac{2}{\kappa}\ln \ln n$ for $x\in [K_i,K_{i+1}]$.

Using~(\ref{exit_probs}) and~(\ref{bound_pot}), we obtain a lower bound: for $\omega \in A(n) \cap F(n) \cap G_1(n)$ we have
\begin{equation}
\label{crossingprob2}
\Po^{K_i}[T_{K_{i+1}}>T_{K_{i-1}}] \geq \frac 1 {K_{i+1}-K_{i-1}} \frac 1 {e^{V(K_{i-1})-V(K_{i+1})}} \geq n^{-(1+2\gamma_0)}.
\end{equation}

During the first~$3n$ steps of the embedded random walk there are two cases, either the walk has reached $n^{\nu}$ or there are at least $n$ steps back. But then if $n^{\nu}$ is reached in less than $3n$ steps, $\Bk(n)$ is stochastically dominated by a $\text{Bin}(3n,n^{-3/2})$ by~(\ref{crossingprob}). Moreover, we get for $f(\cdot)$ such that $f(n)=O(n)$, $\PP$-a.s.\ for~$n$ large enough,
\[
P_{\omega}[\Bk(n)\geq f(n)]  \leq \binom{3n}{n} \Bigl(\frac1 {n^{3/2}}\Bigr)^{n}+P\Bigl[\text{Bin}(3n,n^{-3/2})\geq f(n)\Bigr],
\]
and so using Stirling's formula and Chebyshev's exponential inequality, $\PP$-a.s.\ for~$n$ large enough,
\begin{align}
P_{\omega}[\Bk(n)\geq f(n)]  &\leq \exp(-C_1 n) +C_2\exp(-f(n))
\label{boundback} \\
& \leq C_3\exp(-f(n)). \nonumber
\end{align}

\section{Quenched slowdown}
\label{s_q_slowdown}
In this section, we prove Theorem~\ref{t_q_slow}. Before going into technicalities, 
let us give an informal argument about why we obtain different answers
in Theorem~\ref{t_q_slow}. 
\begin{figure}
\centering
\includegraphics[width=12cm]{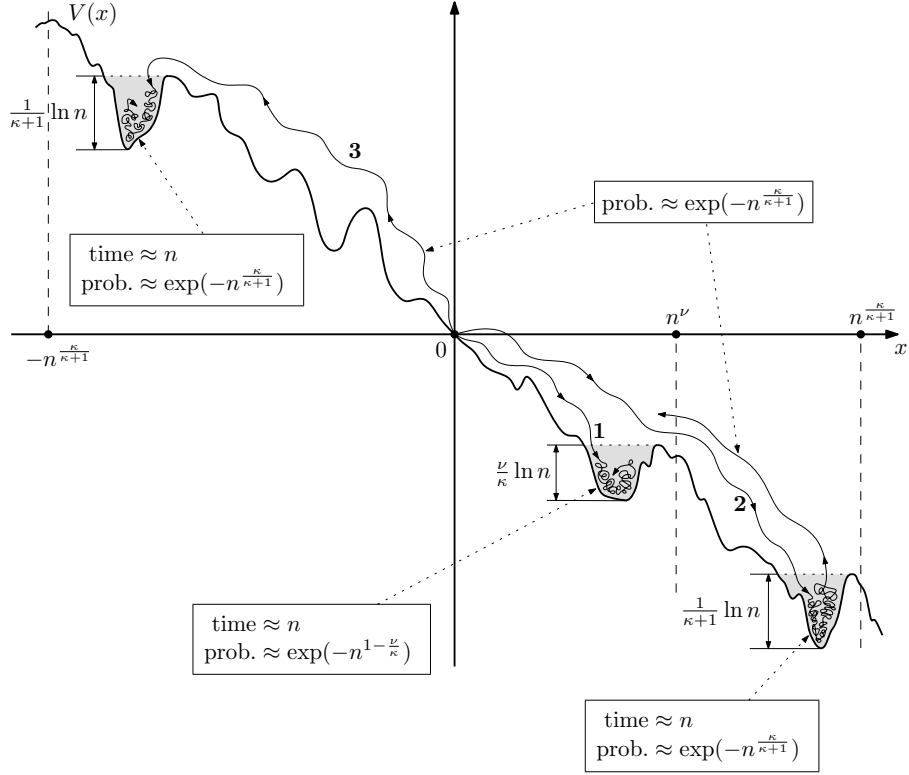}
\caption{The three strategies for the slowdown}
\label{f_strategies}
\end{figure}

Suppose that $\frac{\kappa}{\kappa+1}<1-\frac{\nu}{\kappa}$, or equivalently,
$\nu<\frac{\kappa}{\kappa+1}$. 
Consider the three strategies depicted on Figure~\ref{f_strategies}:
\begin{itemize}
 \item[{\bf 1:}] The particle goes to the biggest valley in the interval
$[0,n^\nu]$, and stays there up to time~$n$.
 \item[{\bf 2:}] The particle goes to the biggest valley in the interval
$[0,n^{\frac{\kappa}{\kappa+1}}]$, stays there up to time~$n-n^{\frac{\kappa}{\kappa+1}}$,
and then goes back to the interval $[0,n^\nu]$.
 \item[{\bf 3:}] The particle goes to the biggest valley in the interval
$[-n^{\frac{\kappa}{\kappa+1}},0]$ (so that typically it has to go
roughly $n^{\frac{\kappa}{\kappa+1}}$ units to the left), and stays there up to time~$n$.
\end{itemize}
By Lemmas~\ref{depthvalley} and~\ref{depthvalley2}, the biggest valley
in the interval $[0,n^\nu]$ has depth of approximately~$\frac{\nu}{\kappa}\ln n$.
Using Proposition~\ref{Lboundconf}, we obtain that the probability of staying there up to time~$n$
is roughly $\exp(-n^{1-\frac{\nu}{\kappa}})$. As for the strategy~{\bf 2},
analogously we find that the biggest valley
in the interval $[0,n^{\frac{\kappa}{\kappa+1}}]$ has depth around~${\frac{1}{\kappa+1}}\ln n$,
and the probability of staying there is roughly $\exp(-n^{\frac{\kappa}{\kappa+1}})$. Then,
the probability of backtracking is again around $\exp(-n^{\frac{\kappa}{\kappa+1}})$.
The situation with the strategy~{\bf 3} is the same as that with strategy~{\bf 2}
(for the strategy~{\bf 3}, we first have to backtrack and then to stay in the valley,
but the probabilities are roughly the same).

So, in the case $\nu<\frac{\kappa}{\kappa+1}$ the strategies~{\bf 2} and~{\bf 3}
are better than the strategy~{\bf 1}. The only situation when we cannot use
neither~{\bf 2} nor~{\bf 3} is when the RWRE has reflection in the origin, and
we are considering the hitting times.

\subsection{Time spent in a valley}\label{s_time_crossing}
We have
\begin{proposition}
\label{crossingtime}
There exists ${\gamma}_4>0$ such that for $\PP$-almost all $\omega$, for all~$n$ large enough we have for $i\leq 2n+1$ and $u\geq 1$,
\begin{align*}
 \Po^{K_i}\bigl[T_{K_{i+1}}>u \bigl({\gamma}_4(\ln n)^{10}e^{H_{i-1}\vee H_i}\bigr)\mid T_{K_{i+1}}<T_{K_{i-1}}\bigr] &\leq e^{-u},\\
\Po^{K_i}\bigl[T_{K_{i-1}}>u \bigl({\gamma}_4(\ln n)^{10}e^{H_{i-1}\vee H_i}\bigr)\mid T_{K_{i-1}}<T_{K_{i+1}}\bigr] &\leq e^{-u}.
\end{align*}
\end{proposition}

\begin{proof}
We prove only the second part of the proposition, the first one uses the same arguments. First, we have
\[
\max_{x\in (K_{i-1},K_{i+1})} \Bigl(\max_{y \in [x,K_{i+1})} V(y) -\min_{y\in [K_{i-1},x)} V(y)\Bigr)=H_{i-1}\vee H_i.
\]

Using~(\ref{crossingprob2}) (or~(\ref{crossingprob}) for the first part of the proposition), we obtain $\PP$-a.s.\ for~$n$ large enough,
\begin{align*}
& \lefteqn{\Po^{K_i}\bigl[T_{K_{i-1}}>u \bigl({\gamma}_4(\ln n)^{10}e^{H_{i-1}\vee H_i}\bigr)\mid
 T_{K_{i+1}}>T_{K_{i-1}}\bigr]} \\
&\leq n^{1+2\gamma_0}\Po^{K_i}\bigl[T_{\{K_{i-1},K_{i+1}\}}>u \bigl({\gamma}_4(\ln n)^{10}e^{H_{i-1}\vee H_i}\bigr),T_{K_{i+1}}>T_{K_{i-1}}\bigr].
 \end{align*}
To estimate this last probability, we may consider the random walk reflected at $K_{i-1}$ and $K_{i+1}$.  On $A(n)$ we have $K_{i+1}-K_{i-1} \leq 2 (\ln n)^2$ and on $G_1(n)\cap F(n)$ we have $\max_{y \in [K_{i-1},K_{i+1}]} V(y) -\min_{y\in [K_{i-1},K_{i+1})} V(y)\leq 2\gamma_0 \ln n$ by~(\ref{bound_pot}). Hence for $n$ such that $\gamma_0\leq (\ln n)^2$ we can apply Proposition~\ref{spectral} with $a=K_{i-1}$, $c=K_{i+1}$, $\tilde{M}\leq 2 (\ln n)^2$ and $H=H_{i-1}\vee H_i$ to get
\begin{align*}
\lefteqn{ P_{\hat\omega}^{K_i}\left[T_{\{K_{i-1},K_{i+1}\}}>u \left({\gamma}_4(\ln n)^{10}e^{H_{i-1}\vee H_i}\right)\right]}\hphantom{***********}\\
&\leq \exp\bigl(-u{\gamma}_4 (\ln n)^2/(32{\gamma}_1)\bigr),
\end{align*}
where $\hat\omega$ denotes the environment with reflection at $K_{i-1}$ and $K_{i+1}$, so that
\begin{align*}
 \lefteqn{\Po^{K_i}\bigl[T_{K_{i-1}}>u \bigl({\gamma}_4(\ln n)^{10}e^{H_{i-1}\vee H_i}\bigr)\mid
 T_{K_{i+1}}>T_{K_{i-1}}\bigr]} \hphantom{*****}\\
 &\leq \exp\bigl(-u{\gamma}_4 (\ln n)^2/(32{\gamma}_1)+(1+2\gamma_0) \ln n\bigr)\\
 &\leq e^{-u},
 \end{align*}
for ${\gamma}_4>32{\gamma}_1((1+2\gamma_0)+1)$ and~$n$ large enough.
\end{proof}

Let  $Z_i$ be a random variable with the same law as $T_{K_{i+1}}$ under $\Po^{K_i}[~\cdot\mid T_{K_{i+1}}<T_{K_{i-1}}]$. Then, 
for $i \in N(-n^{a},n^{b})$ and $H=\max_{i\in N(-n^{a},n^{b})} H_i$, we have that
 $\PP$-a.s.\ for~$n$ large enough 
\begin{equation}
\label{domstoch_1}
\frac{Z_i}{{\gamma}_4e^H(\ln n)^{10}}\prec 1 + \mathbf{e},
\end{equation}
where $\mathbf{e}$ is an exponential random variable with parameter~$1$.
Since $\omega\in G_1(n^{a\vee b})$ $\PP$-a.s.\ for~$n$ large enough, there is a constant $\gamma > 0$ (depending only on $\kappa$) such that
\begin{equation}
\label{domstoch}
\frac{Z_i}{{\gamma}_4n^{(a \vee b)/\kappa}(\ln n)^{{\gamma}}}\prec 1 + \mathbf{e}.
\end{equation}
The same inequality is true when $K_{i-1}$ and $K_{i+1}$ are exchanged. We point out that the same stochastic domination holds in the reflected case, even for $T_{\tilde K_{i_0+2}}$ under $\Pto^{\tilde K_{i_0+1}}[~\cdot\mid T_{\tilde K_{i_0+2}}<T_{\tilde K_{i_0}}]=\Pto^{\tilde K_{i_0+1}}[~\cdot~]$ in which case it is a direct consequence of Proposition~\ref{spectral}.

Using the same kind of arguments as in the proof of Proposition~\ref{crossingtime} we obtain
\begin{proposition}
\label{initcrossingtime}
There exists a positive constant ${\gamma}_4$ (without restriction of generality, the same as in Proposition~\ref{crossingtime}) such that 
for $\PP$-almost all $\omega$, we have for all~$n$ large enough, with $i_0=\card N_n(-n,0)$ and $u\geq 1$,
\begin{align*}
&\Po\bigl[T_{K_{i_0+1}(n)}>u \bigl({\gamma}_4(\ln n)^{10}e^{H_{i_0-1}\vee H_{i_0}}\bigr)\mid T_{K_{i_0+1}(n)}<T_{K_{i_0-1}(n)}\bigr] \leq e^{-u},\\
&\Pto\bigl[T_{K_{i_0+1}(n)}>u \bigl({\gamma}_4(\ln n)^{10}e^{H_{i_0-1}\vee H_{i_0}}\bigr)\bigr] \leq e^{-u}.
\end{align*}
\end{proposition}

Similarly we obtain
\begin{proposition}
\label{rightcrossingtime}
There exists a positive constant ${\gamma}_4$ (without restriction of generality, the same as in Proposition~\ref{crossingtime}) such that
for $\PP$-almost all $\omega$, we have for all~$n$ large enough with $i_1=\card N_n(-n,n^{\nu})$ and $u\geq 1$,
\[
\Po^{K_{i_1}}\bigl[T_{n^{\nu}}>u \bigl({\gamma}_4(\ln n)^{10}e^{H_{i_1-1}\vee H_{i_1}}\bigr)\mid T_{n^{\nu}}<T_{K_{i_1}(n)}\bigr] \leq e^{-u}.
\]
and
\[
\Po^{K_{i_1}}\bigl[T_{K_{i_1-1}}>u \bigl({\gamma}_4(\ln n)^{10}e^{H_{i_1-1}\vee H_{i_1}}\bigr)\mid T_{K_{i_1 -1}}<T_{n^\nu}\bigr] \leq e^{-u}.
\]\end{proposition}

This proposition implies that
\begin{equation}
\label{lastdomstoch}
\frac{\tau_-^{\text{last}}}{{\gamma}_4n^{\nu/\kappa}(\ln n)^{{\gamma}}}\prec 1 + \mathbf{e}.
\end{equation}

\subsection{Time spent for backtracking}
\label{s_time_backtracking}

Recalling the definitions (\ref{defTbacktrack}) and (\ref{Bdef}), we obtain, for the reflected case,
\begin{proposition}
\label{timebacktrack}
For $0<a<b<c<1$, we have $\PP$-a.s.\ for~$n$ large enough,
\[
\Pto\Bigl[\frac{\Tback}{{\gamma_4 n^{\nu/\kappa}(\ln n)^{{\gamma}}}} \geq n^c , \Bk(n) \in[n^a,n^b)\Bigr]\leq \exp(-n^c/4),
\]
where ${\gamma}$ is as in (\ref{domstoch}).
\end{proposition}

\begin{proof}
On the event $\{\Bk(n)\in [n^a,n^b)\}$, we have $\sum_{i\in N(0,n^{\nu})} \xi^{\nu}(i)= \Bk(n)< n^b$, so we can use~(\ref{domstoch}) and~(\ref{lastdomstoch}) to get that $\PP$-a.s.\ for~$n$ large enough,
\begin{equation}
\label{Gammaest}
\frac{\Tback}{\gamma_4 n^{\nu/\kappa} (\ln n)^{{\gamma}}} \prec 2n^b + \text{Gamma}(2n^b,1).
\end{equation}
(note that $\Tback$ is the time spent in valleys from $0$ to $n^\nu$ because we have a reflection at 0). The factor $2$ arises from the fact that each backtracking creates one right-left crossing and one left-right crossing. We use the following bound on the tail of $\text{Gamma}(k,1)$:
\begin{equation}
\label{tail_gamma} P[\text{Gamma}(k,1) \geq u]\leq e^{-u/2} E[\exp(\text{Gamma}(k,1)/2)]=e^{-u/2}2^k.
\end{equation}

Hence we have $\PP$-a.s.\ for~$n$ large enough,
\[
\Pto\Bigl[\frac{\Tback}{{\gamma_4 n^{\nu/\kappa}(\ln n)^{{\gamma}}}} \geq n^c , \Bk(n) \in [n^a,n^b)\Bigr]\leq P[\text{Gamma}(2n^b,1)\geq n^c-2n^b],
\]
and since $(n^c-2n^b)/2-2n^b\ln 2\geq n^c/4$ for $n$ large enough, we conclude with (\ref{Gammaest}).
\end{proof}

In the same way, we get, still for the reflected case
\begin{proposition}
\label{timebacktrackleft}
For $0<a<b<c<1$, we have $\PP$-a.s.\ for~$n$ large enough,
\[
\Pto\Bigl[\frac{\Tleft}{{\gamma_4 n^{\nu/\kappa}(\ln n)^{{\gamma}}}} \geq n^c , \Bk(n) \in[n^a,n^b)\Bigr]\leq \exp(-n^c/4),
\]
where ${\gamma}$ only depends on $\kappa$.
\end{proposition}

\begin{proof}
On the event $\{\Bk(n) \in[n^a,n^b)\}$, $\Tleft$ is lower than the time spent in the valleys of indexes $i_0$ and $i_0+1$ during backtrackings from $\widetilde{K}_{i_0+1}$ to $\widetilde{K}_{i_0}$. Since, there are at most $n^b$ backtracks for this valley and since~(\ref{domstoch}) is valid even for $T_{K_{i_0+2}}$ under $\Pto^{K_{i_0+1}}[~\cdot~]$, we can use the same argument as in the proof of Proposition~\ref{timebacktrack}.
\end{proof}

Next, recalling the definition (\ref{defTbacktrack}), we obtain
\begin{proposition}
\label{timebacktrack1}
For $0<a<b<1$ and $c\in(b\vee \nu,1)$, we have $\PP$-a.s.\ for~$n$ large enough,
\[
\Po\Bigl[\frac{\Tback}{{n^{(b\vee \nu)/\kappa}(\ln n)^{{\gamma}}}} \geq n^c , \Bk(n)\in [n^a,n^b) \Bigr]\leq \exp(-n^c/4),
\]
where ${\gamma}$ only depends on $\kappa$.
\end{proposition}

\begin{proof}
On the event $\{\Bk(n) \in[n^a,n^b)\}$, $\Tback$ consists of the time spent in the valleys indexed by $N_n(-n^b,n^{\nu})$, once this is noted we use the same argument as in the proof of Proposition~\ref{timebacktrack}.
\end{proof}

\subsection{Time spent for the direct crossing}
\label{s_direct_crossing}

We can control $\Tdir$ with the following proposition. Recall (\ref{def_b_prime}) and (\ref{Gdef}).
\begin{proposition}
\label{timedirect}
For all $m\geq m_0(\kappa,\nu)$, we have for $n$ large enough
\[
\Po\left[\Tdir\geq n \right]\leq C(m)\exp(-n^{1-(1+2/m)\frac{\nu}{\kappa}}).
\]
\end{proposition}

\begin{proof}
Recall the definition (\ref{defTdirect}) and let us take $\omega \in B'(n,\nu,m)\cap G_1(n)$. Let us introduce for $k=-1,\ldots,m$,
\begin{align}
N(k)= &\card \{i \in N(-n^{\nu},n^{\nu}): \ H_i \geq \frac{\nu k}{\kappa m} \ln n + 2\ln \ln n \},\\
\sigma(k)=&\card\Bigl\{ i \leq T_{n^{\nu}}: X_i \in \left[K_j(n),K_{j+1}(n)\right)\text{ for some } j \nonumber\\
 & \qquad\text{ with } H_j \in \Bigl[\ln n \frac{\nu k}{\kappa m}+2\ln \ln n, 
\ln n \frac{\nu (k+1)}{\kappa m}+2\ln \ln n\Bigr]\Bigr\}.\label{deftauk}
\end{align}

If $\Tdir \geq n$, then for some $k \in [-1,m]$ the particle spent an amount of time greater than $n/(4m)$ in the valleys of depth 
in $\bigl[\frac{\nu k}{\kappa m} \ln n + 2\ln \ln n, \frac{\nu (k+1)}{\kappa m} \ln n + 2\ln \ln n\bigr]$ because $\omega$ is in $G_1(n)$, so that
\begin{equation}
\label{discretize}
\Po[\Tdir> n] \leq 4m \max_{k \in [-1,m]} \Po[\sigma(k) \geq  n/(4m)].
\end{equation}
 
Using Proposition~\ref{crossingtime}, since $\omega \in B'(n,\nu,m)\cap G_1(n)$, we have $N(k) \leq n^{\nu(1- k/m)}$, and so
\[
\frac{\sigma(k)}{{\gamma}_4(\ln n)^{11}n^{\nu(k+1)/(\kappa m)}} \prec 2n^{\nu(1- k/m)} + \text{Gamma}(2n^{\nu(1- k/m)},1).
\]
For $m>(1-\nu)^{-1}$ we have that $n^{\nu(1-k/m)}=o(n^{1-\nu(k+1)/m}(\ln n)^{-11})$, and for~$n$ large enough (depending on~$\nu$ and~$m$), we use~(\ref{tail_gamma}) to obtain
\begin{align*}
\Po[\sigma(k) \geq  n/(4m)]& \leq P\Bigl[\text{Gamma}(2n^{\nu(1- k/m)},1) \geq  \frac{n^{1-\nu(k+1)/(\kappa m)}}{(\ln n)^{12}} \Bigr]\\
& \leq 4^{n^{\nu(1-k/m)}}\exp\Bigl(-\frac{n^{1-\nu(k+1)/(\kappa m)}}{(\ln n)^{12}}\Bigr)\\
& \leq \exp\Bigl(-2n^{1-\nu(k+2)/(\kappa m)}+\ln 4 n^{\nu(1-k/m)}\Bigr).
\end{align*}

We need to check that $n^{1-(1+2/m)\nu/\kappa}\geq \ln 4 n^{\nu(1-k\epsilon)}$ for any~$k$, 
if we take~$m$ large enough, but this can be done by considering the cases $k=0$ and $k=m$. Hence we get Proposition~\ref{timedirect}.
\end{proof}

\subsection{Upper bound for the probability of quenched slowdown for the hitting time}
\label{s_upper_q_slowdown}

In this section we suppose that $\omega \in A(n)\cap G_1(n) \cap B'(n,\nu,m)$, which is satisfied $\PP$-a.s.\ for~$n$ large enough. First, we consider RWRE with reflection at the origin.
Because of (\ref{decompos})
\begin{align}
\label{decompoUB}
\Pto\left[T_{n^{\nu}}>n\right] \leq &\Pto\left[\Tdir\geq n/5 \right]+\Pto\left[\Tback\geq n/5\right]+\Pto\left[\Tinit\geq n/5\right] \\ \nonumber
                                    &+\Pto\left[\Tright\geq n/5\right]+\Pto\left[\Tleft\geq n/5\right].
\end{align}

Let $\epsilon>0$ and recall (\ref{Bdef}), then
\begin{align*}
\Pto\left[\Tback\geq n/5 \right] \leq & \Pto[\Bk(n)> n^{1-(1+2/m)\nu/\kappa}]\\
&+\Pto[\Tback\geq n/5,\Bk(n)\leq n^{1-(1+2/m)\nu/\kappa}].
\end{align*}

Using (\ref{boundback}), we can write
\[
\Pto[\Bk(n)> n^{1-(1+2/m)\nu/\kappa}]\leq C_2 \exp(-n^{1-(1+2/m)\nu/\kappa}),
\]
 and for~$n$ large enough by Proposition~\ref{timebacktrack},
\begin{align*}
\lefteqn{\Pto[\Tback\geq n/5,\Bk(n)\leq n^{1-(1+2/m)\nu/\kappa}]}\\
&\leq \Pto\Bigl[\frac{\Tback}{n^{\nu/\kappa}(\ln n)^{{\gamma}}}\geq n^{1-(1+1/m)\nu/\kappa} ,\Bk(n)\leq n^{1-(1+2/m)\nu/\kappa}\Bigr] \\
&\leq  \exp(-n^{1-(1+1/m)\nu/\kappa}/4)\\
&\leq \exp(-n^{1-(1+2/m)\nu/\kappa}),
\end{align*}
so 
we obtain
\begin{equation}
\label{backtrackUB}
\Pto\left[\Tback\geq n/5\right] \leq \exp(-n^{1-(1+2/m)\nu/\kappa}).
\end{equation}

By Proposition~\ref{initcrossingtime}, recalling (\ref{defTinit}), we have
\begin{equation}
\label{initUB}
\Pto\left[\Tinit\geq n/5\right] \leq \exp(-n^{1-(1+2/m)\nu/\kappa}).
\end{equation}

Recalling~\ref{defTright}, using Proposition~\ref{rightcrossingtime} and the fact that $\omega \in G_1(n)$, we get
\begin{equation}
\label{rightUB}
\Pto\left[\Tright\geq n/5\right] \leq \exp(-n^{1-(1+2/m)\nu/\kappa}).
\end{equation}

Finally, using (\ref{decompoUB}), (\ref{backtrackUB}), (\ref{initUB}), (\ref{rightUB}) and Proposition~\ref{timedirect}, we get that for all $\epsilon>0$
\[
\Pto \left[T_{n^{\nu}}>n\right] \leq C_3\exp(-n^{1-(1+2/m) \nu/ \kappa}).
\]
Hence, letting~$m$ go to~$\infty$ we obtain
\begin{equation}
\label{UboundQST}
\liminf_{n\to\infty} \frac{\ln (- \ln \Pto \left[T_{n^{\nu}}>n\right])}{\ln n} \geq 1- \frac{\nu}{\kappa},\qquad \text{$\PP$-a.s.}
\end{equation}

Now, we consider RWRE without reflection. All estimates remain true except  (\ref{backtrackUB}) for $\Tback$. Concerning the estimates on $\Tleft$ it is easy to see that since $\{\Tleft>0\}$ implies that $\Bk(n)\geq n/(\ln n)^2-1$, we have using~(\ref{boundback})
\begin{equation}
\label{leftUB2}
\Po[\Tleft\geq n/5] \leq \exp(-n^{1-(1+2/m)\nu/\kappa}).
\end{equation}

It remains to estimate $\Po[\Tback \geq n]$, hence we take $m$ and we note that
\[
\Po[\Tback>n]\leq  \sum_{k=0}^{m}\Po[\Tback>n ,\Bk(n)\in [n^{k/m},n^{(k+1)/m})].
\]

Using~(\ref{boundback}), we obtain that $\PP$-a.s.\ for~$n$ large enough,
\[
\Po[\Tback>n , \Bk(n)\in [n^{k/m},n^{(k+1)/m})]\leq C_3\exp(-n^{k/m}).
\]

Using Proposition~\ref{timebacktrack1}, we obtain that
\[
\Po[\Tback>n , \Bk(n)\in [n^{k/m},n^{(k+1)/m})]\leq C_4\exp(-C_5n^{1-(\nu\vee((k+1)/m))/\kappa}).
\]
Hence, with these estimates on $\Tback$, (\ref{decompoUB}), (\ref{initUB}), (\ref{leftUB2}), (\ref{rightUB}) and Proposition~\ref{timedirect} we obtain that $\PP$-a.s.\ for~$n$ large enough,
\[
\liminf_{n\to\infty} \frac{\ln( - \ln \Po[T_{n^{\nu}}>n])}{\ln n} \geq \min_{k \in[-1,m+1]} \Bigl(\frac{k}{m} \vee \Bigl(1-\frac{\nu \vee ((k+1)/m)}{\kappa}\Bigr)\Bigr),
\]
minimizing we obtain,
\[
\liminf_{n\to\infty} \frac{\ln( - \ln \Po[T_{n^{\nu}}>n])}{\ln n} \geq \Bigl(1- \frac{\nu}{\kappa}\Bigr) \wedge \frac{\kappa}{\kappa+1} -\frac{2}{(1\wedge\kappa) m}, \qquad \text{$\PP$-a.s.}
\]

Taking the limit as~$m$ goes to infinity yields the upper bound in (\ref{nonrefl_hit_q_sl}), 
i.e.,
\begin{equation}
\label{loboT}
 \liminf_{n\to\infty} \frac{\ln(-\ln \Po[T_{n^{\nu}}>n])}{\ln n}
 \geq \Bigl(1-\frac{\nu}{\kappa}\Bigr) \wedge \frac{\kappa}{\kappa+1},
\qquad \text{$\PP$-a.s.}
\end{equation}

\subsection{Upper bound for the probability of quenched slowdown for the walk}
\label{s_inversion_q_sl}
The argument of this section applies for both reflected and non-reflected RWREs,
for the proof in the reflected case, just replace ``$\Po$'' with ``$\Pto$''. We assume that $\omega \in A(n) \cap G_1(n) \cap B'(n,\nu,m)$ which is satisfied $\PP$-a.s.\ for~$n$ large enough.

Set $m\in \Z^+$, we have using Markov's property
\begin{align}
\label{decompo1}
\Po[X_n<n^{\nu}]\leq& \sum_{k=0}^{m}\Po[T_{n^{\nu+(k-1)/m}}<n]\\
 & \qquad\times\max_{i\leq n} \Po^{n^{\nu+(k-1)/m}}[X_i<n^{\nu},~T_{n^{\nu+k/m}}>n-i].
\nonumber
\end{align}

First let us notice that
\begin{align}
\lefteqn{
\max_{i\leq n} \Po^{n^{\nu+(k-1)/m}}[X_i\leq n^{\nu},~T_{n^{\nu+k/m}}>n-i]}\nonumber\\
& \leq\Bigl( \max_{i\leq n} \Po^{n^{\nu+(k-1)/m}}[X_i< n^{\nu}]\Bigr)\wedge\Po^{n^{\nu+(k-1)/m}}[T_{n^{\nu+k/m}}>n].
\label{decompo0}
\end{align}

Using reversibility we have for any $x \in \Z$ (omitting integer parts for simplicity),
\[
\Po^{n^{\nu+(k-1)/m}}[X_i=x] \leq  \frac{\pi(x)}{\pi(n^{\nu+(k-1)/m})},
\]
hence
\[
\max_{i\leq n} \Po^{n^{\nu+(k-1)/m}}[X_i< n^{\nu}]\leq 1\wedge \frac{\pi([-n,n^{\nu}])}{\pi(n^{\nu+(k-1)/m})}.
\]

Recall~(\ref{indices}), then by (\ref{pidef}) and the definition of $b_i$ we get $\pi(b_{i_1})\leq 2 e^{-V(b_{i_1})}$ and  
\[
\pi(b_{i_1})\leq 2 e^{-V(b_{i_1})}\leq C_6 (\ln n)^{2/\kappa} n^{1/\kappa} e^{-V(K_{i_1+1}(n))},
\]
since, due to (\ref{Gdef}), the increase of potential in a valley is at most $\frac 1{\kappa}(\ln n + 2\ln \ln n)$. Hence, 
using~(\ref{weightbottom}) and the fact that the width of the valleys is at most $(\ln n)^2$, we get that
\[
\pi([-n,n^{\nu}])\leq C_7(\ln n)^{2+2/\kappa}n^{1/\kappa} e^{-V(K_{i_1+1}(n))}.
\]
Furthermore, denoting by~$i_2$ the index of the valley containing $n^{\nu+(k-1)/m}$, for $n$ large enough we have using~(\ref{pidef})
\[
\pi(n^{\nu+(k-1)/m}) \geq \pi(K_{i_2-1}(n)),
\]
since on the event $G(n)$ both $V(K_{i_2-1}(n))$ and $V(K_{i_2-1}(n)-1)$ are bigger than $V(n^{\nu+(k-1)/m})$ and $V(n^{\nu+(k-1)/m}-1)$.

On $A(n)$, we have $\abs{(i_2-1)-i_1}\geq \abs{n^{\nu+(k-1)\epsilon}-n^{\nu}}/(\ln n)^2-2$. Since $V(K_i)-V(K_{i+1}) \geq 1/(1\wedge \kappa) \ln n$ 
for $\omega \in G_1(n)$, we have for $k\geq 2$
\begin{align}
\label{decompo2}
\frac{\pi([-n,n^{\nu}])}{\pi(n^{\nu+(k-1)/m})} &\leq C_8 (\ln n)^{2+2/\kappa} n^{1/\kappa} \exp(-(V(K_{i_1+1})-V(K_{i_2-1}))) \\ \nonumber
 &\leq C_9 (\ln n)^{2+2/\kappa} n^{1/\kappa} \exp\Bigl(-C_{10} \frac{ n^{\nu+(k-1)/m}-n^{\nu}}{(\ln n)^2}\Bigr).
\end{align}

Moreover, using~(\ref{nonrefl_hit_q_sl}) in the non-reflected case (or~(\ref{UboundQST}) in the reflected case), we have
\[
\Po^{n^{\nu+(k-1)/m}}[T_{n^{\nu+k/m}}>n]\leq  \exp(-n^{(1-(\nu+(k/m))/\kappa)\wedge (\kappa/(\kappa+1)) -1/m}).
\]
Hence, using this last inequality and~(\ref{decompo2}), the inequality~(\ref{decompo1}) becomes
\begin{align*}
\lefteqn{\Po[X_n<n^{\nu}]}\\
&\leq   \max_{k\in[-1, m+1]} \Bigl[ 1\wedge \Bigl[C_9m n^{1/\kappa} (\ln n)^{2+2/\kappa}  \exp\Bigl(-C_{10} \frac{n^{\nu+(k-1)/m}-n^{\nu}}{(\ln n)^2}\Bigr)\Bigr]\\
                     &\qquad\qquad\qquad \wedge \exp(-n^{(1-(\nu+(k/m))/\kappa)\wedge (\kappa/(\kappa+1)) -1/m})\Bigr],
\end{align*}
so that ${\bf P}$-a.s.,
\begin{align*}
\liminf_{n\to\infty} \frac{\ln (-\ln \Po[X_n<n^{\nu}])}{\ln n}
& \geq \min_{k\in[-1,m+1]}\Biggl[ \Bigl({\mathbf 1}\Bigl\{\frac{k-1}m\geq 0\Bigr\}\Bigl(\nu+\frac{k-1} m\Bigr)\Bigr)\\
 & \qquad\qquad\vee \Bigl(\Bigl(1-\frac{\nu+ k/m}{\kappa}\Bigr) \wedge \frac{\kappa}{\kappa+1} -\frac 1m\Bigr)\Biggr].
\end{align*}

Minimizing over $k$, we obtain 
\[
\liminf_{n\to\infty} \frac{\ln (-\ln \Po[X_n<n^{\nu}])}{\ln n} \geq \Bigl(1-\frac{\nu}{\kappa}\Bigr)\wedge \frac{\kappa}{\kappa+1}-\frac{1}{m}, \qquad \text{$\PP$-a.s.}
\]
Letting $m$ goes to infinity, we obtain
\begin{equation}
\label{upper_q_slowdown}
\liminf_{n\to\infty} \frac{\ln (-\ln \Po[X_n<n^{\nu}])}{\ln n} \geq \Bigl(1-\frac{\nu}{\kappa}\Bigr)\wedge \frac{\kappa}{\kappa+1}, 
\qquad \text{$\PP$-a.s.}     
\end{equation}

\subsection{Lower bound for quenched slowdown}
\label{s_lower_q_sl}
In this section we assume $\omega \in A(n) \cap D(n) \cap F(n)$ which is satisfied $\PP$-a.s.\ for~$n$ large enough. First, we consider RWRE with reflection at the origin.

For all $\epsilon>0$, note that for $n$ large enough there is a valley of depth at least 
$\frac{(1-\epsilon) \nu}{\kappa} \ln n$ strictly before level $n^\nu$ and denote by~$i_2$ the index of the first such valley. Hence
\[
\Pto[T_{n^{\nu}}>n] \geq \Pto^{\tilde K_{i_2}}[T_{\tilde K_{i_2+1}+1}>n],
\]
and using Proposition~\ref{Lboundconf} we obtain
\[
\Pto^{\tilde K_{i_2}}[T_{\tilde K_{i_2+1}+1}>n] \geq \exp(-n^{1-(1-\epsilon)\nu/\kappa+\epsilon}).
\]
Letting $\epsilon$ go to $0$, yields
\begin{equation}
\label{LboundQST}
\limsup_{n\to\infty} \frac{\ln(-\ln \Pto[T_{n^{\nu}}>n])}{\ln n} \leq 1- \frac{\nu}{\kappa}.
\end{equation}
This yields the lower bound for the exit time, so, recalling~(\ref{UboundQST}),
we obtain~(\ref{refl_hit_q_sl}).

Now let us deduce the results on the slowdown. Set $a\in[0,\kappa-\nu)$, for $n$ large enough there is a valley of depth $(\nu+(1-\epsilon)a)/\kappa \ln n$ strictly before $n^{\nu+a}$ whose index is denoted $i_3$. One possible strategy for the walk is to enter the $i_2$-th valley at $\tilde K_{i_2}+1\leq n^{\nu+a}$, stay there up to time $n-(n^{\nu+a}-n^{\nu})- (\ln n)^2$, then go to the left up to time $n$. The probability of this event can be bounded from below by
\begin{align*}
\Pto[X_n< n^{\nu}]  \geq& \Pto\left[ T_{n^{\nu+a}}<n/2\right]\min_{j\leq n} \Pto^{\widetilde K_{i_3}+1}\left[T_{\{\widetilde K_{i_3}-1,\widetilde K_{i_3+1}+1\}} >j\right]\\ &\times n^{-(3/\epsilon_0)(n^{\nu+a}-n^{\nu}+(\ln n)^2)}.
\end{align*}
The first term is bigger than~$1/2$ for~$n$ large enough 
(one can see this by using e.g.~(\ref{LboundQST})). 
The second can be bounded by Proposition~\ref{Lboundconf}
\[
\min_{j\leq n} \Pto^{\widetilde K_{i_3}+1}\left[T_{\{\widetilde K_{i_3}-1,\widetilde K_{i_3+1}+1\}} >j\right] \geq \exp(-n^{1-(\nu+(1-\epsilon)a)/\kappa+\epsilon}),
\]
for $n$ large enough. Then, the last term (going left) was dealt with using the fact that $\omega \in F(n)$.

 This yields for any $a\geq 0$,
\[
\limsup_{n\to\infty} \frac{\ln (-\ln \Pto[X_n< n^{\nu}])}{\ln n} \leq \1{a>0}(\nu+a) \vee \Bigl(1-(1-\epsilon)\frac{\nu+a}{\kappa}+\epsilon\Bigr),
\]
and if we choose $a=0\vee(\kappa/(\kappa+1)-\nu)$, we obtain
\[
\limsup_{n\to\infty} \frac{\ln (-\ln \Pto[X_n< n^{\nu}])}{\ln n} \leq \Bigl(1-\frac{\nu}{\kappa}\Bigr) \wedge \frac{\kappa}{\kappa+1} +\frac{2\epsilon}{\kappa}+\epsilon,\qquad \text{$\PP$-a.s.}
\]
Together with~(\ref{upper_q_slowdown}), this yields~(\ref{refl_loc_q_sl}) by letting $\epsilon$ go to $0$.

Now, we consider the case of RWRE without reflection.
Using the same reasoning, we write
\begin{equation}
\label{part_res_q_sl1}
\limsup_{n\to\infty} \frac{\ln(-\ln \Po[T_{n^{\nu}}>n])}{\ln n} \leq 1- \frac{\nu}{\kappa},\qquad \text{$\PP$-a.s.}
\end{equation}

Now we can see that, if we denote by~$i_4$ the index of a valley of depth at least $(1-\epsilon)/(\kappa+1)\ln n$ between $-n^{\kappa/(\kappa+1)}$ and 0, since we are on $D(n)$, we can go to this valley before reaching $n^{\nu}$ and then stay there for a time at least $n$. This yields, 
\[
\Po[T_{n^{\nu}}>n] \geq \Po[T_{-n^{\kappa/(\kappa+1)}}<T_{n^{\nu}}] \Po^{K_{i_4}}[T_{K_{i_4+1}+1}>n],
\]
bounding the first term by the probability of going to the left on the $n^{\kappa/(\kappa+1)}$ first steps, we get using Proposition~\ref{Lboundconf} that for all~$n$ large enough
\[
\Po^0[T_{n^{\nu}}>n] \geq n^{-(3/\epsilon_0)n^{\kappa/(\kappa+1)}} \exp(-n^{1-(1-2\epsilon)/(\kappa+1)}),
\]
and hence
\begin{equation}
\label{part_res_q_sl2}
\limsup_{n\to\infty} \frac{\ln(-\ln  \Po^0[T_{n^{\nu}}>n])}{\ln n} \leq \frac{\kappa}{\kappa+1}+2\frac{\epsilon}{\kappa+1},\qquad \text{$\PP$-a.s.}
\end{equation}
Moreover, it is clear that
\begin{equation}
\label{X<->T}
\Po[X_n<n^{\nu}] \geq \Po [T_{n^{\nu}} > n],
\end{equation}
and letting~$\epsilon$ go to~$0$ in~(\ref{part_res_q_sl2}) and using~(\ref{part_res_q_sl1}) 
and~(\ref{loboT}), we obtain (\ref{nonrefl_hit_q_sl}) and~(\ref{nonrefl_loc_q_sl}).
This finishes the proof of Theorem~\ref{t_q_slow}. \qed

\section{Annealed slowdown}
\label{s_a_sl}

\subsection{Lower bound for annealed slowdown}
\label{s_lower_a_sl}
Let us define the events 
\[
A'(n,\nu,a)=\bigl\{\text{there exists }x \in [-n^{\nu},n^{\nu}]:\  \max_{y\in [x,n^{\nu}]} V(y)-V(x) \geq (1+a)\ln n\bigr\},
\]
and
\[
A_+'(n,\nu,a)=\bigl\{\text{there exists }x \in [0,n^{\nu}]:\  \max_{y\in [x,n^{\nu}]} V(y)-V(x) \geq (1+a)\ln n\bigr\}.
\]

\begin{lemma}
\label{A2}
We have for $a\in (-1,1)$,
\[
 \lim_{n\to\infty}\frac{\ln \PP[A'(n,\nu,a)]}{\ln n}= \lim_{n\to\infty}\frac{\ln \PP[A_+'(n,\nu,a)]}{\ln n} = -(\kappa- \nu)-a\kappa.
\]
\end{lemma}

\begin{proof}
From (\ref{fellerthm}), it is straightforward to obtain that
\begin{align*}
\PP[A_+'(n,\nu,a)]&\leq \PP[A'(n,\nu,a)] \\
&\leq 2n^{\nu} \PP\Bigl[ \max_{i \geq 0} V(i) \geq (1+a) \ln n\Bigr]\\ &=\Theta(n^{\nu-(1+a)\kappa}).
\end{align*}

In order to give the corresponding lower bound, let us define the event 
\[
A_1(n,a)=\bigl\{\text{there exists }k \in[0,(\ln n)^2] \text{ such that } V(k) \geq (1+a)\ln n \bigr\},
\]
we have
\begin{align*}
\PP[A_1(n,a)]
  \geq& \PP\Bigl[ \max_{i \geq 0} V(i) \geq (1+a)\ln n\Bigr] - \PP[V(\ln n)^2>- \ln n] \\
  &-\PP\Bigl[\max_{i \geq (\ln n)^2} V(i)-V((\ln n)^2) > (2+a) \ln n\Bigr]
\nonumber\\
  =& \Theta( n^{-(1+a)\kappa}),
\end{align*}
where we used (\ref{fellerthm}) and a reasoning similar to the proof of Lemma~\ref{widthvalley}. 
Now, we write
\[
\PP[A'(n,\nu,a)] \geq \PP[A_+'(n,\nu,a)]\geq \frac {n^{\nu}} {\lfloor (\ln n)^2\rfloor } \PP[A_1(n)] =\Theta\Bigl( \frac{n^{\nu-(1+a)\kappa}}{(\ln n)^2}\Bigr),
\]
and Lemma~\ref{A2} follows.
\end{proof}

For any $\epsilon>0$, on the event $A_+'(n,\nu,\epsilon)$ there exists a valley $[K_i, K_{i+1}]$ with $V(K_{i+1})- V(b_i) \geq (1+\epsilon)\ln n$ 
contained in $[0, n^{\nu})$ and we denote by $i_5$ its index. Then we have by Proposition~\ref{cost_to_climb}
\[
\Po[T_{n^{\nu}}>n] \geq P_{\omega}^{b_{i_5}}[T_{K_{i_5+1}+1}>n]\geq 1-\gamma_2(1+n)e^{-(1+\epsilon)\ln n}\geq \frac 12
\]
for $n$ large enough. So  
\[
\PA[T_{n^{\nu}}>n] \geq E[\1{A_+'(n,\nu,\epsilon)}\Po[T_{n^{\nu}}>n]]\geq \frac 12 \PP[A_+'(n,\nu,\epsilon)].
\]
Hence we obtain by Lemma~\ref{A2} that for any $\epsilon>0$
\[
\liminf_{n\to\infty} \frac{\ln \PA[T_{n^{\nu}}>n]}{\ln n} \geq  -(\nu-\kappa) -\kappa \epsilon.
\]
Using~(\ref{X<->T}), we obtain the corresponding lower bound for $\PA[X_n<n^{\nu}]$
as well. Replacing $\Po$ by $\Pto$ and $\PA$ by $\PAt$,
exactly the same argument can be used to obtain the result in the reflected case.

\subsection{Upper bound for annealed slowdown}
\label{s_upper_a_sl}
We prove the upper bound in the non-reflected case, the reflected case follows easily; indeed a simple coupling argument shows that $T_{n^{\nu}}$ in the environment $\tom$ is stochastically dominated by $T_{n^{\nu}}$ in the environment $\omega$. For $m \in \N$ such that $1/m\in (0,\nu)$, we have
\[
\PA[T_{n^{\nu}}>n] \leq \PP[A'(n,\nu,-1/m)]+\EE\big(\1{A'(n,\nu,-1/m)^c}\Po^0[T_{n^{\nu}}>n]\big).
\]
The second term can be further bounded by
\begin{align*}
\lefteqn{\EE\big(\1{A'(n,\nu,-1/m)^c}\Po^0[T_{n^{\nu}}>n]\big)}\\
& \leq \PP[A(n)^c \cup B'(n,\nu,m)^c] \\
& \qquad +\EE\big(\1{A'(n,\nu,-1/m)^c \cap A(n)\cap B'(n,\nu,m)}\Po^0[T_{n^{\nu}}>n]\big),
\end{align*}
where $B'(n,\nu,m)$ is defined in~(\ref{def_b_prime}).

Using Lemma~\ref{A2} we have that $1/n=o(\PP[A'(n,\nu,-1/m)])$, and thus Lemma~\ref{widthvalley} and Lemma~\ref{depthvalley1} imply that 
\[
\PP[A(n)^c \cup B'(n,\nu,m)^c]=o(\PP[A'(n,\nu,-1/m)]).
\]

We can turn~(\ref{domstoch_1}) into the following, for $i \in N(-n^{\epsilon},n^{\nu})$ we have
\[
\text{on $A'(n,\nu,-1/m)^c \cap A(n) \cap B'(n,\nu,m)$,} \qquad \frac{Z}{C_8n^{(1-1/m)}(\ln n)^{{\gamma}}}\prec 1 + \mathbf{e},
\]
where $Z$ has the same law as $T_{K_{i+1}(n)}$ under $\Po^{K_i(n)}[~\cdot\mid T_{K_{i+1}(n)}<T_{K_{i-1}(n)}]$; ${\gamma}={\gamma}(\kappa)$ and $\mathbf{e}$ denotes an exponential random variable of parameter $1$. The same inequality is true when $K_{i-1}(n)$ and $K_{i+1}(n)$ are exchanged. 

This stochastic domination is the key argument for Section~\ref{s_upper_q_slowdown}. We can adapt the proof of Proposition~\ref{timebacktrack}, so that on $A'(n,\nu,-1/m)^c \cap A(n) \cap B'(n,\nu,m)$ we obtain for all $u \geq 1$, 
\[
\Po\Bigl[\frac{\Tback}{{n^{1-1/m}(\ln n)^{{\gamma}}}} \geq \exp(n^{1/(2m)}) , \Bk(n) \leq n^{1/(4m)} \Bigr]\leq e^{-n^{1/(2m)}/4},
\]
and
\[
\Po\Big[\Tright> \frac n5\Big]\leq C_1 \exp(-n^{1/(4m)}).
\]

Moreover, (\ref{boundback}) still holds, so that
\[
\Po[\Bk(n) \geq n^{1/(4m)} ]\leq C_2 \exp(-n^{1/(4m)}),
\]
which yields
\[
\Po\Big[\Tleft> \frac n5\Big]\leq C_3 \exp(-n^{1/(4m)}).
\]

Finally, recalling~(\ref{defTinit}) and using Proposition~\ref{spectral} on $A'(n,\nu,-1/m)^c \cap A(n)$, we obtain
\[
\Po\Big[\Tinit> \frac n5\Big]\leq C_4 \exp(-n^{1/(4m)}).
\]

Since Proposition~\ref{timedirect} remains true and $A'(n,\nu,-1/m)^c \subset G(n)$, we get that for all $\omega \in A'(n,\nu,-1/m)^c \cap A(n) \cap B'(n,\nu,m)$
\[ 
\Po[T_{n^{\nu}}>n] \leq C_5\exp(-n^{1/(4m)}).
\]
Loosely speaking it costs at least $\exp(-n^{1/(2m)})$ to backtrack $n^{1/m}$ times, hence, on $A'(n,\nu,-1/m)^c \cap A(n) \cap B'(n,\nu,m)$, we can only see valleys of size lower than $(1-1/m) \ln n$. To spend a time $n$ in those valleys would cost at least $\exp(-n^{1/(2m)})$. This finally implies that for all $m>0$,
\[
\limsup_{n\to \infty} \frac{\ln \EE\Bigl[\1{A'(n,\nu,-1/m)^c,A(n)^c,B'(n,\nu,m)^c}\Po^0[T_{n^{\nu}}>n]\Bigr]}{\ln n} = 0,
\]
so that
\begin{equation}
\label{UBAS}
\limsup_{n\to\infty} \frac{\ln \PA[T_{n^{\nu}}>n]}{\ln n} \leq -(\kappa- \nu) +\frac{\kappa}m,
\end{equation}
the result for the hitting time follows by letting $m$ go to infinity.

It is simple to extend this result to the position of the walk, indeed if $X_n<n^{\nu}$ then $T_{n^{(1+1/m)\nu}}>n$ or $\Bk(n) \geq n^{1/(2m)}$ and hence using~(\ref{boundback}) , we get for all $m>0$
\[
\PA[X_n<n^{\nu}]\leq \PA[T_{n^{(1+1/m)\nu}}>n]+C_6e^{-n^{1/(2m)}},
\]
and the result follows by using~(\ref{UBAS}) and letting $m$ go to infinity.

This concludes the proof of Theorem~\ref{t_a_slow}. \qed

\section{Backtracking}
\label{s_backtracking}
In this section we prove Theorem~\ref{t_backtrack}.

\subsection{Quenched backtracking for the hitting time}
\label{s_q_backtracking_t}

Set $\nu \in (0,1)$ and consider $\Po[T_{-n^{\nu}}<n]$. First, we get that
\[
\text{for all $\omega\in F(n)$,} \qquad \Po[T_{-n^{\nu}}<n]\geq n^{-(3/\epsilon_0)n^{\nu}},
\]
since the particle can go straight to the left during the first~$n^{\nu}$ steps, hence
\begin{equation}
\label{backtrack_limsup}
\limsup_{n\to \infty}  \frac{\ln(-\ln \Po[T_{-n^{\nu}}<n])}{\ln n} \leq \nu.
\end{equation}

Secondly, we remark that if $(-\infty,-n^{\nu}]$ 
has been hit before time~$n$ then, at some time $i\leq n$ 
the particle is at $X_i\in [-n, -n^{\nu}]$ and hence
for all $\omega$ 
\begin{align}
\Po[T_{-n^{\nu}}<n]&\leq \sum_{i=1}^n \Po[X_i \in [-n,-n^{\nu}] ] \nonumber\\
&\leq n \max_{i\leq n}\Po[X_i \in [-n,-n^{\nu}] ].
\label{back_pi}
\end{align}

In order to estimate this quantity, we use arguments similar to those in Section~\ref{s_inversion_q_sl}, i.e., first we use the reversibility of the walk to write
\[
\max_{i\leq n}\Po[X_i \in [-n,-n^{\nu}]]\leq \frac{\pi([-n,-n^{\nu}])}{\pi (0)},
\]
then, the right-hand side can be estimated in the same way as we obtained~(\ref{decompo2}), 
and so we get on $A(n)\cap G_1(n)$ that
\[
\frac{\pi([-n,-n^{\nu}])}{\pi (0)}\leq C_1 (\ln n)^{2+2/\kappa} n^{1/\kappa}\exp(-C_2 n^{\nu}/(\ln n)^2).
\]

The previous inequality and~(\ref{back_pi}) yield
\[
\text{for all $\omega\in A(n)\cap G(n)$,} \qquad \Po[T_{-n^{\nu}}<n]\leq C_3 n^{1+2/\kappa}\exp(-C_2 n^{\nu}/(\ln n)^2),
\]
so that
\[
\liminf_{n\to \infty}  \frac{\ln(-\ln \Po[T_{-n^{\nu}}<n])}{\ln n} \geq \nu.
\]

Together with ~(\ref{backtrack_limsup}), this  proves~(\ref{backtrack_T}).

\subsection{Quenched backtracking for the position of the random walk}
\label{s_q_backtracking_rw}

Let us denote $a_0=\frac{\kappa}{\kappa+1} \vee \nu$. We give a lower bound for $\Po[X_n<-n^{\nu}]$. For~$n$ large enough, there exists ${\bf P}$-a.s.\ a valley of depth $(1-\epsilon)(a_0/\kappa)\ln n$ of index~$i_2$, between $-n^{a_0}$ and $0$. Consider the event that the walker goes to this valley directly and stays there up to time $n-n^{a_0}$ and then goes to the left for the next~$n^{a_0}+1$ steps. On this event we have $X_n< -n^{a_0}$, so we obtain
\begin{align*}
 \Po[X_n<-n^{\nu}] &\geq n^{-(3/\epsilon_0)2(n^{a_0}+1)} \Po^{K_{i_2+1}-1}[T_{\{K_{i_2}-1,K_{i_2+1}+1\}} \geq n]\\
& \geq n^{-(3/\epsilon_0)2(n^{a_0}+1)}\exp(-n^{1-(1-2\epsilon)a_0/\kappa}),
\end{align*}
where we used Proposition~\ref{Lboundconf} and $\omega\in F(n)$. Hence we obtain
\[
\limsup_{n\to \infty} \frac{\ln(-\ln\Po[X_n<-n^{\nu}])}{\ln n} \leq a_0 +\frac{2\epsilon a_0}{\kappa},
\]
and letting $\epsilon$ go to 0 we have
\begin{equation}
\label{limsup_q_back_rw}
\limsup_{n\to \infty} \frac{\ln(-\ln\Po[X_n<-n^{\nu}])}{\ln n} \leq a_0.
\end{equation}

Turning to the upper bound, we have for $m\in \N$,
\begin{equation}
\label{back_step1}
\Po[X_n<-n^{\nu}]\leq \sum_{k=0}^{m}\Po[T_{n^{(k-1)/m}}<n] \max_{i\leq n} \Po^{n^{(k-1)/m}}[T_{n^{k/m}}>n-i,X_i< -n^{\nu}],
\end{equation}
where once again
\begin{align*}
 \lefteqn{
\max_{i\leq n} \Po^{n^{(k-1)/m}}[T_{n^{k/m}}>n-i,X_i< -n^{\nu}]}\\
& \leq \Bigl(\max_{i\leq n} 
   \Po^{n^{(k-1)/m}}[X_i< -n^{\nu}]\Bigr)\wedge\Po^{n^{(k-1)/m}}[T_{n^{k/m}}>n] .
\end{align*}

First, using~(\ref{nonrefl_hit_q_sl}), for~$n$ large enough
\begin{equation}
\label{back_step2}
\Po^{n^{(k-1)/m}}[T_{n^{k/m}}>n ] \leq \exp(-n^{(1-(k/m)/\kappa)\wedge (\kappa/(\kappa+1))-1/m}).
\end{equation}

Then, as in Section~\ref{s_inversion_q_sl}, the reversibility of the walk yields that
\begin{equation}
\label{back_step3}
\max_{i\leq n}\Po^{n^{(k-1)/m}}[X_i \in [-n,-n^{\nu}]]\leq \frac{\pi([-n,-n^{\nu}])}{\pi (n^{(k-1)/m})},
\end{equation}
the right-hand side can be estimated in the same way we obtained~(\ref{decompo2}) and we get on $A(n)\cap G(n)$
\begin{equation}
\label{back_step4}
\frac{\pi([-n,-n^{\nu}])}{\pi (n^{(k-1)/m})}\leq C_4 \exp(-C_5 (n^{(k-1)/m}+n^{\nu})/(\ln n)^2).
\end{equation}

Putting together~\eqref{back_step1},~\eqref{back_step2},~\eqref{back_step3}, and~\eqref{back_step4}, we obtain
\begin{align*}
\lefteqn{\liminf_{n\to \infty} \frac{\ln(-\ln\Po[X_n<-n^{\nu}])}{\ln n}}\\
 & \geq
\min_{k\in[0,m]} \Bigl(\Bigl(\Bigl(1-\frac{k}{m\kappa} \Bigr)\wedge \frac{\kappa}{\kappa+1}\Bigr)\vee \Bigl(\frac{k-1}m \vee \nu \Bigr) \Bigr)-\frac 1m,
\end{align*}
minimizing yields that
\[
\liminf_{n\to \infty} \frac{\ln(-\ln\Po[X_n<-n^{\nu}])}{\ln n} \geq a_0 -\frac 2{m},
\]
letting~$m$ go to infinity and recalling~\eqref{limsup_q_back_rw} we obtain~\eqref{q_backtrack}.




\subsection{Annealed backtracking}
\label{s_a_backtracking}
Let $\theta_0=\EE\left[\ln\rho_0\right]<0$. Define
\[
 \RR = \Bigl\{\omega: V(x)\leq \frac{\theta_0}{3}n^{\nu} \text{ for }x\in [0,n],
   |V(x)+\theta_0x|\leq \frac{|\theta_0|}{3}n^\nu
\text{ for }x\in[-n^\nu,0)\Bigr\}.
\]
Since~$V$ is a sum of i.i.d.~random variables having some finite exponential moments, we can use large deviations techniques to obtain~$C_6$ such that
\begin{equation}
\label{bound_PP_R}
 \PP[\RR] \geq 1-2ne^{-C_6n^\nu}.
\end{equation}

Then, on~$\RR$, using~(\ref{exit_probs}), we obtain
\begin{align}
 \Po[T_{-n^\nu}<n] &\leq \Po[T_{-n^\nu}<T_n] \nonumber\\
 &\leq C_7n \exp\Bigl(-\frac{2\theta_0}{3}n^\nu\Bigr). \label{Po_on_RR}
\end{align}
Using~(\ref{bound_PP_R}) and~(\ref{Po_on_RR}), we
obtain
\begin{equation}
\label{b*}
 \PA[X_n<-n^\nu] \leq  \PA[T_{-n^\nu}<n]\leq e^{-C_8 n^\nu}.
\end{equation}

On the other hand, we easily obtain that 
\begin{equation}
\label{b**}
\PA[T_{-n^{\nu}}<n]\geq \PA[X_n<-n^\nu] \geq \Bigl(\frac{\hdel}2\Bigl)^{n^\nu} n^{-C_9},
\end{equation}
where $\hdel>0$ is such that $\PP[1-\omega_0 \geq \hdel]>1/2$. Indeed on the event of probability at least $(1/2)^{n^{\nu}}$ that $1-\omega_x\geq \hdel$ for $x\in (-n^{\nu},0]$, the particle can go ``directly'' (to the left on each step) to $(-n^\nu)$,
and then the cost of creating a valley of depth $2\ln n$ there is polynomial and then it costs nothing to stay there for a time~$n$ by Proposition~\ref{cost_to_climb}.
Now, (\ref{b*}) and~(\ref{b**}) imply~(\ref{a_backtrack}).
This finishes the proof of Theorem~\ref{t_backtrack}. \qed

\section{Speedup}
\label{s_speedup}
In this section we prove Theorem~\ref{t_speedup}.
So, we have $\kappa<1$, $\nu\in (\kappa,1)$; let us denote 
$g(\alpha)=\nu+\frac{\alpha}{\kappa}-\alpha$, and let $\alpha_0=\kappa\frac{1-\nu}{1-\kappa}$.
Clearly, $g(\alpha)$ is a linear function, $g(0)=\nu<1$, $g(\nu)=\frac{\nu}{\kappa}>1$,
and $g(\alpha_0)=1$; note also that $\nu-\alpha_0=\frac{\nu-\kappa}{1-\kappa}$.

The discussion in this section is for the RWRE on~$\Z$ (i.e., without reflection),
the proof for the reflected case is quite analogous.

\subsection{Lower bound for the quenched probability of speedup}
\label{s_lower_q_speedup}
We are going to obtain a lower bound for $\Po[X_n> n^\nu]$.

By Lemma~\ref{depthvalley1} and Borel-Cantelli, for any fixed~$m$, $\omega\in B'(n,\alpha_0,m)\cap A(n)\cap F(n)$
for all~$n$ large enough, $\PP$-a.s.\ 
(recall the definition of~$A(n)$ and~$B'(n,\alpha_0,m)$ from Section~\ref{s_estimates_env}).
So, from now on we suppose that $\omega\in B'(n,\alpha_0,m)\cap A(n)$.

Let us denote $M=N_n(0,n^{\nu})$, define the index sets
\begin{align*}
 \II_0 &= \{i\in M : H_{i-1}\vee H_i \leq \ln\ln n\},\\
 \II_k &= \Bigl\{i\in M : (H_{i-1}\vee H_i) - \ln\ln n \in \Bigl[\frac{(k-1)\alpha_0}{m\kappa}\ln n,
          \frac{k\alpha_0}{m\kappa}\ln n\Bigr)\Bigr\}
\end{align*}
for $k\in[1,m-1]$, and
\begin{align*}
 \UU &= \Bigl\{i\in M : H_{i-1}\vee H_i \geq \frac{(m-1)\alpha_0}{m\kappa}\ln n + \ln\ln n\Bigr\}.
\end{align*}
Note that on~$B'(n,\alpha_0,m)$
\begin{align}
 \card\UU &\leq n^{\nu-\alpha_0 + \frac{\alpha_0}{m}} 
             = n^{\frac{\nu-\kappa}{1-\kappa} + \frac{\alpha_0}{m}}, \label{razmer_U}\\
 \card\II_k &\leq n^{\nu-\frac{k\alpha_0}{m}}, \qquad \text{ for all } k=1,\ldots,m-1.\label{razmer_I_k}
\end{align}

Recalling~(\ref{indices}) we define the quantities $\sigma_{i_0}=T_{K_{i_0+1}}$, $\sigma_{i_1}=T_{n^\nu}-T_{K_{i_1}}$,
and $\sigma_j=T_{K_{j+1}}-T_{K_j}$ for $j=i_0+1,\ldots,i_1-1$.
Then for $\epsilon>0$, we can write
\begin{align}
\Po[X_n> n^{(1-\epsilon)\nu}] \geq & \Po\Bigl[\sum_{k=0}^{m-1}\sum_{i\in\II_k}\sigma_i\leq \frac{n}{2}\Bigr]
     \Po\Bigl[\sum_{i\in\UU}\sigma_i\leq \frac{n}{2}\Bigr]\nonumber\\
 & {}\times   \Po^{n^\nu}\bigl[X_j> n^{(1-\epsilon)\nu} \text{ for all }
       i\in[0,n-n^\nu]\bigr].\label{decompos_low_P_speed}
\end{align}

Let us obtain lower bounds for the three terms in the right-hand side
of~(\ref{decompos_low_P_speed}). First, we write using~(\ref{razmer_I_k})
\begin{align}
 \Po\Bigl[\sum_{k=0}^{m-1}\sum_{i\in\II_k}\sigma_i\leq \frac{n}{2}\Bigr] &\geq
  \prod_{k=0}^{m-1} \Po\Bigl[\sum_{i\in\II_k}\sigma_i\leq \frac{n}{2m}\Bigr]\nonumber\\
 &\geq \prod_{k=0}^{m-1} \Po\Bigl[\sigma_i\leq \frac{1}{2m}n^{1-(\nu-\frac{k\alpha_0}{m})} 
 \text{ for all }i\in\II_k\Bigr]. \label{coisa*}
\end{align}
Now, consider any $\ell\in\II_k$ and write
\begin{align*}
\lefteqn{\Po\Bigl[\sigma_\ell\leq \frac{1}{2m}n^{1-(\nu-\frac{k\alpha_0}{m})} \Bigr]}\qquad\\
 &\geq
  \Po^{K_\ell}[T_{K_{\ell+1}}<T_{K_{\ell-1}}] \\ & \quad\times
   \Po^{K_\ell}\Bigl[T_{\{K_{\ell-1},K_{\ell+1}\}}\leq \frac{1}{2m}n^{1-(\nu-\frac{k\alpha_0}{m})}  \mid T_{K_{l+1}}<T_{K_{l-1}}\Bigr].
\end{align*}

By the formula~(\ref{crossingprob}), on~$A(n)$ we have
\[
 \Po^{K_\ell}[T_{K_{\ell+1}}<T_{K_{\ell-1}}] \geq 1 - n^{-3/2},
\]
and by Proposition~\ref{crossingtime},
\begin{align*}
 \Po^{K_\ell}\Bigl[T_{\{K_{\ell-1},K_{\ell+1}\}}\leq \frac{1}{2m}n^{1-(\nu-\frac{k\alpha_0}{m})}\mid T_{K_{l+1}}<T_{K_{l-1}}\Bigr]
 \\ \geq 1-\exp\Bigl(-\frac{C_1}{m(\ln n)^{{\gamma}}}n^{1-(\nu-\frac{k\alpha_0}{m})-\frac{k\alpha_0}{m\kappa}}\Bigr),
\end{align*}
so
\begin{equation}
\label{coisa**}
 \Po\Bigl[\sigma_\ell\leq \frac{1}{2m}n^{1-(\nu-\frac{k\alpha_0}{m})} \Bigr] \geq
   (1 - n^{-3/2}) 
    \Bigl(1-\exp\Bigl(-\frac{C_1}{m(\ln n)^{{\gamma}}}
          n^{1-g(\frac{k\alpha_0}{m})}\Bigr)\Bigr).
\end{equation}

Now, for $k\leq m-1$ we have 
\[
 1-g\Bigl(\frac{k\alpha_0}{m}\Bigr)\geq \frac{(1-\kappa)\alpha_0}{m\kappa},
\]
so~(\ref{coisa*}) and~(\ref{coisa**}) imply that 
\begin{align}
 \Po\Bigl[\sum_{k=0}^{m-1}\sum_{i\in\II_k}\sigma_i\leq \frac{n}{2}\Bigr] &\geq 
   \prod_{k=0}^{m-1} \Biggl[  (1 - n^{-3/2})\Bigr) 
        \Bigl(1-\exp\Bigl(-\frac{C_1}{m(\ln n)^{{\gamma}}}
          n^{\frac{(1-\kappa)\alpha_0}{m\kappa}}\Bigr)\Bigr) \Biggr]^{n^\nu}\nonumber\\
 &\to 1 \qquad \text{as }n\to\infty. \label{ocenka_term1}
\end{align}

Now, we obtain a lower bound for the second term in the right-hand side of~(\ref{decompos_low_P_speed}).
On $G_1(n)$, we get an upper bound on $\rho_i$ for $i\in[-n,n]$ and hence we have $\omega_x \geq n^{-C_2}$, we obtain for any $\ell\in\UU$ (imagine that, to cross the corresponding interval, the particle just goes to the right at each step)
\begin{equation}
 \Po\Bigl[\sigma_\ell \leq \frac{1}{2}n^{1-(\nu-\alpha_0)-\frac{\alpha_0}{m}}\Bigr] 
          \geq n^{-C_2 (\ln n)^2},
\end{equation}
so, 
\begin{align}
 \Po\Bigl[\sum_{i\in\UU}\sigma_i\leq \frac{n}{2}\Bigr] &\geq 
   \Po\Bigl[\sigma_\ell \leq \frac{1}{2}n^{1-(\nu-\alpha_0)-\frac{\alpha_0}{m}}
          \text{ for all }\ell\in\UU\Bigr] \nonumber\\
 &\geq \bigl(n^{-C_2(\ln n)^2}\bigr)^{n^{\frac{\nu-\kappa}{1-\kappa}+\frac{\alpha_0}{m}}}\nonumber\\
 & = \exp\Bigl(- C_2(\ln n)^3 n^{\frac{\nu-\kappa}{1-\kappa}+\frac{\alpha_0}{m}}\Bigr)
\label{ocenka_term2}
\end{align}
(recall that $\nu-\alpha_0 = \frac{\nu-\kappa}{1-\kappa}$).

As for the third term in~(\ref{decompos_low_P_speed}), using~(\ref{exit_probs}) we
easily obtain that, on~$A(n)\cap G(n)$, 
\begin{equation}
\label{ocenka_term3}
 \Po^{n^\nu}\bigl[X_j> n^{(1-\epsilon)\nu} \text{ for all } j\in[0,n-n^\nu]\bigr]  \geq \Po^{n^\nu}[T_n<T_{n^{(1-\epsilon)\nu}}]>C_3>0.
\end{equation}

Now, plugging (\ref{ocenka_term1}), (\ref{ocenka_term2}), and~(\ref{ocenka_term3})
into~(\ref{decompos_low_P_speed}) and sending~$m$ to~$\infty$, we obtain that
\[
\limsup_{n\to\infty} \frac{\ln(-\ln \Po[X_n> n^{(1-\epsilon)\nu}])}{\ln n} \leq \frac{\nu-\kappa}{1-\kappa}, \qquad \text{$\PP$-a.s.}
\]
applying this for $\nu'=\nu/(1-\epsilon)$ and letting $\epsilon$ go to 0,
\begin{equation}
\label{lower_b_sp}
\limsup_{n\to\infty} \frac{\ln(-\ln \Po[X_n> n^{\nu'}])}{\ln n} \leq \frac{\nu'-\kappa}{1-\kappa}, \qquad \text{$\PP$-a.s.}
\end{equation}
Since obviously $\Po[T_{n^\nu}< n]\geq\Po[X_n> n^\nu]$, (\ref{lower_b_sp}) holds
for $\Po[T_{n^\nu}< n]$ as well.

\subsection{Upper bound for the quenched probability of speedup}
\label{s_upper_q_speedup}
Fix~$\epsilon>0$ such that $\alpha_0+\epsilon<\nu$. Define 
\begin{align*}
 \WW &= \Bigl\{i\in N_n(0,n^\nu) : H_i\geq \frac{\alpha_0+\epsilon}{\kappa}\ln n - 4\ln\ln n\Bigr\},\\
 \Psi_n^\epsilon &= \Bigl\{\omega: \card\WW \geq \frac{1}{3} n^{\nu-\alpha_0-\epsilon}\Bigr\}. 
\end{align*}
By Lemma~\ref{depthvalley2}, on each subinterval of length~$n^{\alpha_0+\epsilon}$ we find a valley
of depth at least $\frac{\alpha_0+\epsilon}{\kappa}\ln n - 4\ln\ln n$ with probability
at least $1/2$. Since the interval $[0,n^\nu]$ contains $n^{\nu-\alpha_0-\epsilon}$
such subintervals, we have
\begin{equation}
\label{PP_Theta}
 \PP[\Psi_n^\epsilon] \geq 1 - \exp(-C_4n^{\nu-\alpha_0-\epsilon}),
\end{equation}
in particular by Borel-Cantelli's Lemma, $\PP$-a.s.~we have $\omega \in \Psi_n^\epsilon$ for $n$ large enough.

For $i\in\WW$, define $\ts_i=T_{K_{i+1}+1}-T_{K_i+1}$, and let 
\[
 s_0 = \frac{1}{4{\gamma}_2(\ln n)^4} n^{\frac{\alpha_0+\epsilon}{\kappa}}.
\]
Then, by Proposition~\ref{cost_to_climb}, for any $i\in\WW$,
\begin{align}
 \Po[\ts_i<s_0] &\leq 2{\gamma}_2 s_0 \exp\Bigl(-\frac{\alpha_0+\epsilon}{\kappa}\ln n + 4\ln\ln n\Bigr)
   \nonumber\\
 &= 2{\gamma}_2 s_0n^{-\frac{\alpha_0+\epsilon}{\kappa}}(\ln n)^4 \nonumber\\
 &= \frac{1}{2}. \label{ocenka_climb}
\end{align}

Define the family of random variables $\zeta_i=\1{\ts_i<s_0}$, $i\in\WW$.
These random variables are independent with respect to~$\Po$, and
$\Po[\zeta_i=1]\leq 1/2$ by~(\ref{ocenka_climb}). Suppose without
restriction of generality that (recall that $g(\alpha_0)=1$)
\[
 \frac{1}{3}s_0 \times \frac{1}{3}n^{\nu-\alpha_0-\epsilon}
 = \frac{1}{36{\gamma}_2 (\ln n)^4} n^{g(\alpha_0+\epsilon)} > n.
\]
Then, since $\card\WW \geq \frac{1}{3}n^{\nu-\alpha_0-\epsilon}$ for $\omega \in \Psi_n^\epsilon$, we see using large deviations techniques that for $n$ large enough
\begin{align}
 \Po[T_{n^\nu}<n] &\leq \Po\Bigl[\sum_{i\in\WW}\zeta_i > \frac{2}{3} \card\WW\Bigr]\nonumber\\
 &\leq \exp\bigl(-C_5 n^{\frac{\nu-\kappa}{1-\kappa}-\epsilon}\bigr) \label{oc_up_q_sp}
\end{align}
(recall that $\nu-\alpha_0 = \frac{\nu-\kappa}{1-\kappa}$).
Since~$\epsilon>0$ is arbitrary, we obtain 
\begin{equation}
\label{upper_b_sp}
 \liminf_{n\to\infty} \frac{\ln(-\ln \Po[T_{n^\nu}<n])}{\ln n} \geq \frac{\nu-\kappa}{1-\kappa} \qquad \text{$\PP$-a.s.}
\end{equation}
Together with~(\ref{lower_b_sp}), this shows~(\ref{eq_q_speed}).

\subsection{Annealed speedup}
\label{s_a_speedup}
As usual, the quenched lower bound obtained in Section~\ref{s_lower_q_speedup} also yields the
annealed one, i.e. (\ref{lower_b_sp}) implies that
\begin{equation}
\limsup_{n\to\infty} \frac{\ln(-\ln \PA[X_n> n^\nu])}{\ln n} \leq \frac{\nu-\kappa}{1-\kappa}, 
\end{equation}
Turning to the upper bound, we have by~(\ref{PP_Theta})
and~(\ref{oc_up_q_sp}) that
\begin{align*}
 \PA[T_{n^\nu}<n] &= \int \Po[T_{n^\nu}<n] \, d\PP\\
 &\leq \int_{\Psi_n^\epsilon} \Po[T_{n^\nu}<n] \, d\PP + \PP[(\Psi_n^\epsilon)^c]\\
 &\leq \exp\bigl(-C_5 n^{\frac{\nu-\kappa}{1-\kappa}-\epsilon}\bigr)
            + \exp\bigl(-C_4 n^{\frac{\nu-\kappa}{1-\kappa}-\epsilon}\bigr),
\end{align*}
and this implies~(\ref{eq_a_speed}). This finishes the proof of Theorem~\ref{t_speedup}. \qed

\section*{Acknowledgements}
A.F. would like to thank the ANR \lq\lq MEMEMO\rq\rq, the \lq\lq Accord France-Br\'esil\rq\rq and the ARCUS program.

S.P.\ is thankful to FAPESP (04/07276--2), CNPq (300328/2005--2 and 471925/2006--3), 
and N.G. and S. P. are thankful to CAPES/DAAD (Probral) for financial support.

We thank two anonymous referees whose extremely careful lecture of the first version lead to many improvements.

\end{document}